\documentclass{article}
\usepackage{graphicx}
\usepackage{amsmath}
\usepackage{amssymb,amsthm}
\usepackage[colorlinks=true]{hyperref}
\usepackage{pdfsync}

\usepackage{graphics} 
\usepackage{epsfig} 
\usepackage{epstopdf}
\usepackage{subfigure}
\usepackage{color}

\graphicspath{{./Figures/}}

\usepackage{amssymb,mathtools}
\usepackage{amsthm}

\newtheorem{thm}{Theorem}

\newtheorem{assum}{Assumption}
\newtheorem{lem}{Lemma}
\newtheorem{cor}{Corollary}
\newtheorem{rem}{Remark}

\title{Boundary output feedback stabilization of a cascade of $N$ heat equations}

\author{Hugo Lhachemi, Christophe Prieur, and Emmanuel Tr{\'e}lat
\thanks{Hugo Lhachemi is with Universit{\'e} Paris-Saclay, CNRS, CentraleSup{\'e}lec, Laboratoire des signaux et syst\`emes, 91190, Gif-sur-Yvette, France (email: hugo.lhachemi@centralesupelec.fr).}
\thanks{Christophe Prieur is with Universit\'e Grenoble Alpes, CNRS, Gipsa-lab, 38000 Grenoble, France (e-mail: christophe.prieur@gipsa-lab.fr).}
\thanks{Emmanuel Trélat is with Sorbonne Universit\'e, Universit\'e Paris Cit\'e, CNRS, Inria, Laboratoire Jacques-Louis Lions, LJLL, F-75005 Paris, France (e-mail: emmanuel.trelat@sorbonne-universite.fr).}}

\date{}

\begin{document}

\maketitle

\begin{abstract}
This paper solves the problem of output feedback stabilization for a cascade of $N$ heat equations that are coupled at the boundary, the input being a scalar boundary control applied to the first heat equation of the cascade, and the scalar output being either a distributed or a pointwise in-domain measurement done on the last equation of the cascade. Two different configurations are studied in details. The first one consists of a cascade of $N$ heat equations with totally disconnected spectra. The second one consists of a cascade of $N$ identical heat equations, inducing eigenvalues of multiplicity $N$. In both cases, the problem is solved thanks to a spectral analysis and a study of the modal controllability and observability properties. The key point is that the generalized eigenvectors form a Riesz basis of the state space. The stabilization property is established in $L^2$ and $H^1$ norms.
\end{abstract}

\section{Introduction}

The design of control strategies for partial differential equation (PDE) cascades remains challenging. The case of coupled hyperbolic-parabolic PDE systems has been studied in \cite{chen2017backstepping,ghousein2020backstepping} using backstepping transformations. Control of a heat-wave PDE boundary coupling has been investigated in \cite{zhang2004polynomial,lhachemi2025controllability}. Boundary null controllability of two coupled parabolic PDEs has been studied in \cite{bhandari2021boundary}. Hyperbolic-elliptic couplings have been studied in \cite{chowdhury2023boundary,rosier2013unique}. Cascades of heat-heat equations have been reported in \cite{wang2015stabilization,kang2016stabilisation,tang2024boundary,tang2025sampled,lhachemi2025boundary}.  The situation becomes even more stringent for cascades of $N$ PDEs. In \cite{auriol2020output,auriol2024output} the authors deal with cascades of $N$ interconnected hyperbolic systems using backstepping-based methods.

To our knowledge, few contributions have addressed the topic of an arbitrary number $N$ of coupled heat equations. The observer design for a system of $N$ heat equations with in-domain and boundary couplings has been studied in \cite{camacho2020boundary}. The state-feedback boundary control of $N$ in-domain coupled heat equations, by means of $N$ boundary controls, has been reported in \cite{vazquez2016boundary}. Finally, the state-feedback boundary control of a chain of $N$ heat-ODE cascades has been studied in \cite{xu2023stabilization} for PDEs that cascade into the ODEs through bounded operators. 

The present paper reports one of the first contributions regarding the boundary output feedback stabilization of a cascade of $N$ heat equations by means of one scalar boundary control and one scalar measurement. We consider the system
\begin{subequations}\label{eq: studied PDE cascade}
\begin{align}
	& y_{t}^j = y_{xx}^j + a_j y^j , && 1 \leq j \leq N, \\ 
	& y_x^j(t,0) = y^{j+1}(t,1) ,\quad y_x^j(t,1) = 0 , && 1 \leq j \leq N-1, \\
	& y_x^N(t,0) = u(t) , \quad y_x^N(t,1) = 0, && \\
	& y^j(0,x) = y_0^j(x) , && 1 \leq j \leq N,
\end{align}
\end{subequations}
where $a_1,\ldots,a_N \in\mathbb{R}$. Here, $u(t)$ stands for the control input at time $t$, and $y_0^j \in L^2(0,1)$, $1 \leq j \leq N$, represent the initial conditions. Two different configurations are studied. The first one consists of a cascade of $N$ heat equations with two-by-two distinct eigenvalues. At the opposite, the second one consists of a cascade of $N$ identical heat equations, inducing eigenvalues of multiplicity $N$. In both cases, the problem is solved by developing spectral reduction techniques (see \cite{russell1978controllability,coron2004global,lhachemi2020pi,trelat2024control}) thanks to a full characterization of the eigenelements of the underlying operator. A key point of our approach is that the study of the eigenelements is not done for each component of the cascade individually, but directly for the full cascade seen as one single system (see \cite{lhachemi2025controllability}). We show that the generalized eigenvectors of the system form a Riesz basis of the state space, allowing us to study the modal controllability and observability properties of the cascade. Based on these results, output feedback control strategies are designed. The output, that can take the form either of a distributed or of a pointwise measurement, is done solely on the last equation of the cascade (described by $y^1$ in \eqref{eq: studied PDE cascade}). We establish stabilization results in $L^2$ and $H^1$ norms.

The paper is organized as follows. The configuration of the PDE cascade \eqref{eq: studied PDE cascade} with two-by-two distinct eigenvalues is studied in Section~\ref{sec: simple eig}. 
The case of $N$ identical heat equations, inducing eigenvalues of multiplicity $N$, is reported in Section~\ref{sec: mult eig}. Concluding remarks are provided in Section~\ref{sec: conclusion}.

\section{$N$ heat equations with two-by-two distinct eigenvalues}\label{sec: simple eig}

In this section, we make the following standing assumption, meaning that the spectra of the coupled heat equations are totally disconnected.

\begin{assum}\label{asum: simple eigenvalues}
Setting $\lambda_{i,k} = a_i - k^2 \pi^2$ and $\Lambda_i = \{ \lambda_{i,k} \,\mid\, k \in\mathbb{N} \}$ for all $1 \leq i \leq N$ and $k \in \mathbb{N}$, we assume that 
$i \neq j \Rightarrow \Lambda_i \cap \Lambda_j = \emptyset$.
\end{assum}

\subsection{Spectral properties}
We define the real Hilbert space
\begin{subequations}\label{eq: Hilbert space H^0}
\begin{equation}
\mathcal{H}^0 =  \big( L^2(0,1) \big)^N
\end{equation}
endowed with its usual inner product
\begin{equation}
\left\langle (f^j)_{1 \leq j \leq N} , (g^j)_{1 \leq j \leq N} \right\rangle
= \sum_{j=1}^N \int_0^1 f^j(x) g^j(x) \,\mathrm{d}x .
\end{equation}
\end{subequations}
We define the unbounded operator $\mathcal{A}:D(\mathcal{A})\rightarrow\mathcal{H}^0$ on
\begin{multline}
D(\mathcal{A}) = \Big\{  (f^j)_{1 \leq j \leq N} \in \mathcal{H}^0 \quad\mid\quad (f^j)_{1 \leq j \leq N} \in \big( H^2(0,1) \big)^N , \quad  (f^N)'(0) = (f^N)'(1) = 0 , \\
 (f^j)'(0) = f^{j+1}(1) ,\quad (f^j)'(1) = 0 ,\quad 1 \leq j \leq N-1 \Big\}  \label{eq: 2b2dev operator A - domain}
\end{multline}
by
$\mathcal{A} \, (f^j)_{1 \leq j \leq N} = \Big((f^j)''+a_j f^j\Big)_{1 \leq j \leq N}$.

\begin{lem}\label{lem: eigenstructures A}
Under Assumption \ref{asum: simple eigenvalues}, the eigenvalues of $\mathcal{A}$ are given by
\begin{equation}\label{eq: lambda_i_k}
\lambda_{i,k} = a_i - k^2 \pi^2 ,\quad 1 \leq i \leq N ,\quad k \in\mathbb{N},
\end{equation}
with associated eigenvectors $\phi_{i,k} = ( \phi_{i,k}^j )_{1 \leq j \leq N} \in D(\mathcal{A})$ given by
\begin{equation}\label{eq: simple eig phi}
\phi_{i,k}^j(x) = \left\{ \begin{array}{ll}
\displaystyle
\frac{(-1)^{k+i-j} \mu_k}{\prod_{j'=j}^{i-1}r_{i,k}^{j'}\sinh(r_{i,k}^{j'})} \cosh(r_{i,k}^{j}(1-x)) &\textrm{if}\ \ 1 \leq j \leq i-1,   \\[5mm]
\varphi_k(x) \triangleq \mu_k \cos(k\pi x) &\textrm{if}\ \ j=i,  \\[2mm]
0 &\textrm{if}\ \ i+1 \leq j \leq N , 
\end{array}\right.
\end{equation}
where $\mu_k = \left\{ \begin{array}{ll} 1 &\textrm{if}\ k=0  \\ \sqrt{2} &\textrm{if}\ k \geq 1 \end{array}\right.$
and $r_{i,k}^{j}$ is one of the two square roots of $\lambda_{i,k}-a_j$.
\end{lem}

\begin{proof}
For a given $c \neq 0$, let us first compute the solution of $f'' + af = \lambda f$ such that $f'(1) = 0$ and $f'(0) = c$ when $\lambda - a \neq - k^2 \pi^2$ for any $k\in\mathbb{N}$. Denoting by $r$ one of the two square roots of $\lambda - a \neq 0$, we have $f(x) = \alpha e^{rx} + \beta e^{-rx}$ for some $\alpha,\beta\in\mathbb{C}$. The boundary conditions imply that
\begin{equation*}
\begin{bmatrix} r & -r \\ r e^r & -r e^{-r} \end{bmatrix}
\begin{bmatrix} \alpha \\ \beta \end{bmatrix}
= \begin{bmatrix} c \\ 0 \end{bmatrix} .
\end{equation*}
The square matrix is invertible if and only if $r^2 \sinh r \neq 0$, that is $r \neq ik\pi$ for any $k\in\mathbb{Z}$. This is satisfied based on our assumption that $r^2 = \lambda - a \neq - k^2 \pi^2$ for any $k\in\mathbb{N}$. Hence $\alpha = - \frac{c}{2r\sinh r} e^{-r}$ and $\beta = - \frac{c}{2r\sinh r} e^{r}$, implying that $f(x) = -\frac{c}{r \sinh r} \cosh(r(1-x))$.
The proof of the lemma then follows by using an induction argument.
\end{proof}

In order to establish that the set $\Phi = \{ \phi_{i,k} \,\mid\, 1 \leq i \leq N ,\; k \in\mathbb{N} \}$ of eigenvectors of $\mathcal{A}$ is a Riesz basis of $\mathcal{H}^0$, we first compute the eigenvectors of the adjoint operator $\mathcal{A}^*$ of $\mathcal{A}$, which is given by 
$\mathcal{A}^* ( (g^j)_{1 \leq j \leq N} ) 
= \Big((g^j)''+a_j g^j\Big)_{1 \leq j \leq N}$
on the domain
\begin{multline}\label{eq: domain A*}
D(\mathcal{A}^*) = \Big\{ (g^j)_{1 \leq j \leq N} \in \mathcal{H}^0 \,\mid\, (g^j)_{1 \leq j \leq N} \in \big( H^2(0,1) \big]^N , \quad (g^1)'(0) = (g^1)'(1) = 0 ,  \\
 (g^j)'(0) = 0 ,\; (g^j)'(1) = - g^{j-1}(0) ,\; 2 \leq j \leq N \Big\} . 
\end{multline}

\begin{lem}
Under Assumption \ref{asum: simple eigenvalues},  the eigenvalues of $\mathcal{A}^*$ are $\lambda_{i,k}$ with associated eigenvectors 
$\psi_{i,k} = \Big( \psi_{i,k}^j \Big)_{1 \leq j \leq N} \in D(\mathcal{A}^*)$ given by
\begin{equation}\label{eq: simple eig psi}
\psi_{i,k}^j(x) = \left\{ \begin{array}{ll}
0 & \textrm{if}\ \ 1 \leq j \leq i-1 , \\
\varphi_k(x) = \mu_k \cos(k\pi x) & \textrm{if}\ \  j=i , \\
\displaystyle \frac{(-1)^{j-i} \mu_k}{\prod_{j'=i+1}^{j}r_{i,k}^{j'}\sinh(r_{i,k}^{j'})} \cosh(r_{i,k}^{j}x) & \textrm{if}\ \ i+1 \leq j \leq N . 
\end{array}\right.
\end{equation}
\end{lem}

\begin{proof}
Given $c \neq 0$, let us compute the solution of $g'' + ag = \lambda g$ such that $g'(0) = 0$ and $g'(1) = c$ when $\lambda - a \neq - k^2 \pi^2$ for any $k \in\mathbb{N}$. Denoting by $r$ one of the two square roots of $\lambda - a \neq 0$, we have $g(x) = \alpha e^{rx} + \beta e^{-rx}$ for some $\alpha,\beta\in\mathbb{C}$. The boundary conditions imply that
\begin{equation*}
\begin{bmatrix} r & -r \\ r e^r & -r e^{-r} \end{bmatrix}
\begin{bmatrix} \alpha \\ \beta \end{bmatrix}
= \begin{bmatrix} 0 \\ c \end{bmatrix} .
\end{equation*}
The square matrix is invertible if and only if $r^2 \sinh r \neq 0$, that is $r \neq ik\pi$ for any $k\in\mathbb{Z}$. This is satisfied based on our assumption that $r^2 = \lambda - a \neq - k^2 \pi^2$ for any $k\in\mathbb{N}$. Hence $\alpha = \beta = \frac{c}{2r\sinh r}$, implying that $g(x) = \frac{c}{r \sinh r} \cosh(rx)$.
The proof of the lemma then follows by using an induction argument.
\end{proof}


\begin{lem}\label{eq: distinct eig - Riesz basis}
$\Phi = \{ \phi_{i,k} \,\mid\, 1 \leq i \leq N ,\; k \in\mathbb{N} \}$ is a Riesz basis of $\mathcal{H}^0$, of dual Riesz basis $\Psi = \{ \psi_{i,k} \,\mid\, 1 \leq i \leq N ,\; k \in\mathbb{N} \}$. Moreover $\mathcal{A}$ generates a $C_0$-semigroup.
\end{lem}

\begin{proof}
We start by noting that $\Phi$ and $\Psi$ are biorthogonal, i.e., $\langle \phi_{i,k} , \psi_{i',k'} \rangle = \delta_{(i,k),(i',k')}$. Indeed, from their expressions, we have $\langle \phi_{i,k} , \psi_{i,k} \rangle = \int_0^1 \phi_{i,k}^i(x) \psi_{i,k}^i(x) \,\mathrm{d}x = \int_0^1 \varphi_{k}(x)^2 \,\mathrm{d}x = 1$. If $(i,k) \neq (i',k')$ then $\lambda_{i,k} \neq \lambda_{i',k'}$, hence $\lambda_{i,k} \langle \phi_{i,k} , \psi_{i',k'} \rangle = \langle \mathcal{A} \phi_{i,k} , \psi_{i',k'} \rangle = \langle \phi_{i,k} , \mathcal{A}^* \psi_{i',k'} \rangle = \lambda_{i',k'} \langle \phi_{i,k} , \psi_{i',k'} \rangle$ implying that $\langle \phi_{i,k} , \psi_{i',k'} \rangle = 0$.

Based on \cite[Chap.~1, Thm.~9]{young2001introduction}, it suffices to show that $\mathrm{span}\,\Phi$ and $\mathrm{span}\,\Psi$ are dense in $\mathcal{H}^0$, and that each of $\Phi$ and $\Psi$ form a Bessel sequence\footnote{A sequence $(\varphi_n)_{n \geq 0}$ of a separable Hilbert space $X$ is said to be a Bessel sequence if $(\langle f , \varphi_n \rangle)_{n \geq 0} \in \ell^2(\mathbb{N})$ for all $f \in X$.}. 

Let us show that $\mathrm{span}\,\Phi$ is dense if $\mathcal{H}^0$. Let $f = (f^i)_{1 \leq i \leq N} \in\mathcal{H}^0$ be such that $\langle f , \phi_{i,k} \rangle = 0$ for all integers $1 \leq i \leq N$ and $k \in\mathbb{N}$. Let us prove that $f = 0$. For $i = 1$ we have $\langle f , \phi_{1,k} \rangle = \int_0^1 f^1(x) \varphi_k(x) \,\mathrm{d}x = 0$ for all $k \in\mathbb{N}$. Since $(\varphi_k)_{k \geq 0}$ forms a Hilbert basis of $L^2(0,1)$, we infer that $f^1 = 0$. Assume next we have shown that $f^1 = f^2 = \cdots = f^l = 0$ for some $1 \leq l \leq N-1$. Then for $i=l+1$ we have $\langle f , \phi_{l+1,k} \rangle = \sum_{j=1}^{l+1} \int_0^1 f^j(x) \phi_{l+1,k}^j(x) \,\mathrm{d}x = \int_0^1 f^{l+1}(x) \varphi_k(x) \,\mathrm{d}x = 0$ for all $k \in\mathbb{N}$. As before, this implies that $f^{l+1} = 0$. By induction, we infer that $f = 0$, giving the claimed conclusion. A similar argument is employed to show that $\mathrm{span}\,\Psi$ is dense if $\mathcal{H}^0$.

In order to prove that $\Phi$ and $\Psi$ are Bessel sequences, we first establish asymptotic estimates of $\phi_{i,k}^j$ for $1 \leq j \leq i-1$ and of $\psi_{i,k}^j$ for $i+1 \leq j \leq N$ as $k \rightarrow + \infty$. All the asymptotic estimates derived below hold for $k \rightarrow + \infty$, uniformly with respect to $i,j$ and $x \in [0,1]$. Recalling that $(r_{i,k}^{j})^2 = \lambda_{i,k} - a_j = - \alpha_{i,j} - k^2 \pi^2$ with $\alpha_{i,j} = a_j - a_i$, for all sufficiently large integers $k$ we have $r_{i,k}^{j} = i \sqrt{k^2 \pi^2 + \alpha_{i,j}}$ with $\sqrt{k^2 \pi^2 + \alpha_{i,j}} \in\mathbb{R}$. Hence 
\begin{subequations}
\begin{equation}\label{eq: r_i_j_k}
r_{i,k}^{j} 
= ik\pi \sqrt{ 1 + \frac{\alpha_{i,j}}{k^2 \pi^2} } 
= i ( k\pi + \beta_{i,k}^{j} )
\end{equation}
where 
\begin{equation}\label{eq: beta_i_k^j}
\beta_{i,k}^{j} = \frac{\alpha_{i,j}}{2k\pi} + \mathrm{O}(1/k^3) \in \mathbb{R} .
\end{equation}
\end{subequations}
We thus infer that
$\sinh (r_{i,k}^{j}) = i (-1)^k \sin(\beta_{i,k}^{j}) = i \frac{(-1)^k \alpha_{i,j}}{2k\pi} + \mathrm{O}(1/k^3)$.
Combining the two latter asymptotic expansions, we have
$r_{i,k}^{j} \sinh (r_{i,k}^{j}) = \frac{(-1)^{k+1}\alpha_{i,j}}{2} + \mathrm{O}(1/k^2)$,
yielding
\begin{equation}
\frac{1}{r_{i,k}^{j} \sinh (r_{i,k}^{j})}  = \frac{1}{\frac{(-1)^{k+1}\alpha_{i,j}}{2} + \mathrm{O}(1/k^2)} 
= \frac{(-1)^{k+1} 2}{ \alpha_{i,j}} + \mathrm{O}(1/k^2) . \label{eq: asympt estimate 1/r/sinh(r)}
\end{equation}
By the definition \eqref{eq: simple eig phi} of $\phi_{i,k}^j$ for $1 \leq j \leq i-1$, we deduce from \eqref{eq: asympt estimate 1/r/sinh(r)} that
\begin{equation}
\prod_{j'=j}^{i-1} \frac{1}{r_{i,k}^{j'} \sinh (r_{i,k}^{j'})} 
= \prod_{j'=j}^{i-1} \left( \frac{(-1)^{k+1} 2}{ \alpha_{i,j'} } + \mathrm{O}(1/k^2) \right) 
= \frac{2^{i-j} (-1)^{(k+1)(i-j)}}{\prod_{j'=j}^{i-1} \alpha_{i,j'}} + \mathrm{O}(1/k^2) . \label{eq: estimate prod 1}
\end{equation}
In view of \eqref{eq: r_i_j_k} we get
$\cosh( r_{i,k}^{j} (1-x) )
= (-1)^k \cos ( k \pi x ) \cos ( \beta_{i,k}^{j} (1-x) ) + (-1)^k \sin ( k \pi x ) \sin ( \beta_{i,k}^{j} (1-x) )$.
Recalling \eqref{eq: beta_i_k^j}, we have
\begin{equation}\label{eq: cos_beta and sin_beta}
\cos ( \beta_{i,k}^{j} (1-x) ) = 1 + \mathrm{O}(1/k^2) , \qquad
\sin ( \beta_{i,k}^{j} (1-x) ) = \frac{\alpha_{i,j}}{2k\pi}(1-x) + \mathrm{O}(1/k^3) , 
\end{equation}
hence
\begin{equation*}
\cosh( r_{i,k}^{j} (1-x) ) = (-1)^k \cos(k\pi x) + \mathrm{O}(1/k) . 
\end{equation*}
Since $\mu_k = \sqrt{2}$ for $k \geq 1$, we deduce from \eqref{eq: simple eig phi} (for $1 \leq j \leq i-1$) that
\begin{multline}
\phi_{i,k}^j(x) 
= \frac{(-1)^{k+i-j} \mu_k}{\prod_{j'=j}^{i-1}r_{i,k}^{j'}\sinh(r_{i,k}^{j'})} \cosh(r_{i,k}^{j}(1-x)) 
= \frac{2^{i-j} (-1)^{k(i-j)}}{\prod_{j'=j}^{i-1} \alpha_{i,j'}} \sqrt{2} \cos(k\pi x) + \mathrm{O}(1/k) \\
= A_{i,k}^{j} \varphi_k(x) + R{\phi}_{i,k}^j(x) \label{eq: estimate phi_i_k^j} 
\end{multline}
with $A_{i,k}^{j} = \frac{2^{i-j} (-1)^{k(i-j)}}{\prod_{j'=j}^{i-1} \alpha_{i,j'}}$ and
\begin{equation}\label{eq: def Rphi}
R{\phi}_{i,k}^j(x) \triangleq \phi_{i,k}^j(x) - A_{i,k}^{j} \varphi_k(x) = \mathrm{O}(1/k) .
\end{equation}
We are now in a position to prove that $\Phi$ is a Bessel sequence. Let $f = (f^j)_{1 \leq j \leq N}$ be fixed. For an arbitrary integer $1 \leq i \leq N$, we infer from \eqref{eq: simple eig phi} (for $i+1\leq j\leq N$) and \eqref{eq: estimate phi_i_k^j} that 
$$
\langle f, \phi_{i,k} \rangle 
= \sum_{j=1}^{i-1} A_{i,k}^{j} \int_0^1 f^j \varphi_k \,\mathrm{d}x 
+ \sum_{j=1}^{i-1} \int_0^1 f^j R{\phi}_{i,k}^j \,\mathrm{d}x + \int_0^1 f^i \varphi_k \,\mathrm{d}x .
$$
Since $(\varphi_k)_{k \geq 0}$ is a Hilbert basis of $L^2(0,1)$, we get that $\Big( \int_0^1 f^j \varphi_k \,\mathrm{d}x \Big)_{k \in\mathbb{N}} \in \ell^2(\mathbb{N})$ for all $1 \leq j \leq i$. Furthermore, since $\vert A_{i,k}^{j} \vert = \frac{2^{i-j}}{\prod_{j'=j}^{i-1} \vert \alpha_{i,j'} \vert}$ is independent of $k$, we deduce that $\Big( A_{i,k}^{j} \int_0^1 f^j \varphi_k \,\mathrm{d}x \Big)_{k \in\mathbb{N}} \in \ell^2(\mathbb{N})$ for all $1 \leq j \leq i-1$. Finally, since $R{\phi}_{i,k}^j(x) = \mathrm{O}(1/k)$, we infer that 
$$
\big\vert \int_0^1 f^j R{\phi}_{i,k}^j \,\mathrm{d}x \big\vert \leq \Vert f^j \Vert_{L^2} \Vert R{\phi}_{i,k}^j \Vert_{L^2} = \mathrm{O}(1/k) ,
$$
hence $\Big( \int_0^1 f^j R{\phi}_{i,k}^j \,\mathrm{d}x \Big)_{k \in\mathbb{N}} \in \ell^2(\mathbb{N})$ for all $1 \leq j \leq i-1$. This shows that $\big( \langle f, \phi_{i,k} \rangle  \big)_{k \in\mathbb{N}} \in \ell^2(\mathbb{N})$. This result holds for all $1 \leq i \leq N$, hence $\big( \langle f, \phi_{i,k} \rangle  \big)_{1 \leq i \leq N, k \in\mathbb{N}} \in \ell^2(\mathbb{N})$, showing that $\Phi$ is a Bessel sequence.

It remains to show that $\Psi$ is also a Bessel sequence. Recalling the definition \eqref{eq: simple eig psi} of $\psi_{i,k}^j$ for $i+1 \leq j \leq N$, we infer from \eqref{eq: asympt estimate 1/r/sinh(r)} that
$$
\prod_{j'=i+1}^{j} \frac{1}{r_{i,k}^{j'} \sinh (r_{i,k}^{j'})} 
= \prod_{j'=i+1}^{j} \left( \frac{(-1)^{k+1} 2}{ \alpha_{i,j'} } + \mathrm{O}(1/k^2) \right) 
= \frac{2^{j-i} (-1)^{(k+1)(j-i)}}{\prod_{j'=i+1}^{j} \alpha_{i,j'}} + \mathrm{O}(1/k^2) .
$$
Owing to \eqref{eq: r_i_j_k}, we have
$\cosh( r_{i,k}^{j} x ) = \cos ( k \pi x ) \cos ( \beta_{i,k}^{j} x ) - \sin ( k \pi x ) \sin ( \beta_{i,k}^{j} x )$.
Recalling \eqref{eq: beta_i_k^j}, we infer that $\cos ( \beta_{i,k}^{j} x ) = 1 + \mathrm{O}(1/k^2)$ and $\sin ( \beta_{i,k}^{j} x ) = \frac{\alpha_{i,j}}{2k\pi}x + \mathrm{O}(1/k^3)$, hence $\cosh( r_{i,k}^{j} x ) = \cos(k\pi x) + \mathrm{O}(1/k)$. Since $\mu_k = \sqrt{2}$ for $k \geq 1$, we deduce that
$$
\psi_{i,k}^j(x) 
= \frac{2^{j-i} (-1)^{k(j-i)}}{\prod_{j'=i+1}^{j} \alpha_{i,j'}} \sqrt{2} \cos(k\pi x) + \mathrm{O}(1/k) 
= A_{i,k}^{j} \varphi_k(x) + R\psi_{i,k}^j(x) 
$$
with $A_{i,k}^{j} = \frac{2^{j-i} (-1)^{k(j-i)}}{\prod_{j'=i+1}^{j} \alpha_{i,j'}}$ and $R\psi_{i,k}^j(x) \triangleq \psi_{i,k}^j(x) - A_{i,k}^{j} \varphi_k(x) = \mathrm{O}(1/k)$. Following now the same procedure as in the previous paragraph, we infer that $\Psi$ is a Bessel sequence. 

Hence $\Phi$ is a Riesz basis, of dual Riesz basis $\Psi$.
Using \eqref{eq: lambda_i_k}, since $\mathcal{A}$ is closed, we infer that $\mathcal{A}$ is a Riesz spectral operator that generates a $C_0$-semigroup (see \cite{curtain2012introduction}).
\end{proof}

The previous result is instrumental for studying the solutions in $L^2$ norm, that is, in the Hilbert space $\mathcal{H}^0$. The following result will be useful for studying the solutions in $H^1$ norm, that is in the Hilbert space 
\begin{subequations}\label{eq: def space H^1}
\begin{equation}
\mathcal{H}^1 = \big( H^1(0,1) \big)^N
\end{equation}
endowed with its natural inner product
\begin{equation}
\left\langle (f^j)_{1 \leq j \leq N} , (g^j)_{1 \leq j \leq N} \right\rangle
= \sum_{j=1}^N \int_0^1 \left( f^j(x) g^j(x) + (f^j)'(x) (g^j)'(x) \right) \mathrm{d}x .
\end{equation}
\end{subequations}

\begin{lem}\label{lem: simple eig Riesz basis H1}
The set $\tilde{\Phi} = \{ \tilde{\phi}_{i,k} \,\mid\, 1 \leq i \leq N ,\; k \in\mathbb{N} \}$, where $\tilde{\phi}_{i,k} = \frac{1}{\sqrt{1+k^2\pi^2}} \phi_{i,k}$, is a Riesz basis of $\mathcal{H}^1$.
\end{lem}

\begin{proof}
We define $\hat{\Phi} = \{ \hat{\phi}_{i,k} \,\mid\, 1 \leq i \leq N ,\; k\in\mathbb{N} \}$ where $\hat{\phi}_{i,k} = ( \hat{\phi}_{i,k}^j )_{1 \leq j \leq N}$ and
\begin{equation*}
\hat{\phi}_{i,k}^j(x) = \left\{ \begin{array}{ll}
A_{i,k}^{j} \varphi_k(x) & \textrm{if}\ \ 1 \leq j \leq i-1 , \\
\varphi_k(x) & \textrm{if}\ \ j=i , \\
0 & \textrm{if}\ \ i+1 \leq j \leq N .
\end{array}\right.
\end{equation*}
We infer from \eqref{eq: estimate phi_i_k^j} that $\Vert \phi_{i,k}^j - \hat{\phi}_{i,k}^j \Vert_{L^\infty} = \mathrm{O}(1/k)$, hence
$\sum_{1 \leq i \leq N} \sum_{k \geq 0} \Vert \phi_{i,k} - \hat{\phi}_{i,k} \Vert_{\mathcal{H}^0}^2 < \infty$.
Since $\hat{\Phi}$ is $\omega$-linearly independent, Bari's theorem (see \cite{gohberg1978introduction}) implies that $\hat{\Phi}$ is a Riesz basis of $\mathcal{H}^0$.

Let $\Delta$ be the Neumann Laplacian operator defined by $\Delta f=f''$ on $D(\Delta)=\{f\in H^2(0,1)\ \mid\ f'(0)=f'(1)=0\}$. Then $\mathrm{id}-\Delta$ is a positive selfadjoint operator, mapping isometrically and surjectively $D(\Delta)$ to $L^2(0,1)$, and thus the mapping $J=\sqrt{\mathrm{id}-\Delta}:H^1(0,1)\rightarrow L^2(0,1)$ is a surjective isometry (note that $\langle J f_1 , J f_2 \rangle = \langle (1-\Delta) f_1 , f_2 \rangle = \int_0^1 (-f_1''+f_1)f_2 \,\mathrm{d}x = \int_0^1 (f_1' f_2' + f_1 f_2) \,\mathrm{d}x = \langle f_1 , f_2 \rangle$). Therefore the mapping
\begin{equation*}
	\begin{array}{ccccc}
		\mathfrak{J} & : & \mathcal{H}^1 & \rightarrow & \mathcal{H}^0 \\
		& & (f^j)_{1 \leq j \leq N} & \mapsto & (J f^j)_{1 \leq j \leq N} .
	\end{array}		 
\end{equation*}
is also a surjective isometry. Hence,  $\check{\Phi} = \mathfrak{J}^{-1} \hat{\Phi} = \{ \check{\phi}_{i,k} \,\mid\, 1 \leq i \leq N ,\; k \in\mathbb{N} \}$ with $\check{\phi}_{i,k} = \mathfrak{J}^{-1}\hat{\phi}_{i,k}^j$ is a Riesz basis of $\mathcal{H}^1$. Noting that $J^{-1}\varphi_k = \frac{1}{\sqrt{1+k^2\pi^2}}\varphi_k$, we have 
\begin{equation*}
\check{\phi}_{i,k}^j(x) = \left\{ \begin{array}{ll}
\frac{A_{i,k}^{j}}{\sqrt{1+k^2\pi^2}} \varphi_k(x) & \textrm{if}\ \ 1 \leq j \leq i-1 , \\
\frac{1}{\sqrt{1+k^2\pi^2}} \varphi_k(x) & \textrm{if}\ \ j=i , \\
0 & \textrm{if}\ \ i+1 \leq j \leq N .
\end{array}\right.
\end{equation*}
To conclude the proof, we apply Bari's theorem by comparing $\tilde{\Phi}$ to the Riesz basis $\check{\Phi}$ of $\mathcal{H}^1$. Since $\tilde{\Phi}$ is $\omega$-linearly independent, it remains to show that
\begin{multline}\label{eq: cond app Bari's theorem}
\sum_{1 \leq i \leq N} \sum_{k \in\mathbb{N}} \Vert \tilde{\phi}_{i,k} - \check{\phi}_{i,k} \Vert_{\mathcal{H}^1}^2 
= \sum_{1 \leq i \leq N} \sum_{1 \leq j \leq i-1} \sum_{k \in\mathbb{N}} \Vert \tilde{\phi}_{i,k}^j - \check{\phi}_{i,k}^j \Vert_{H^1}^2 \\
= \sum_{1 \leq i \leq N} \sum_{1 \leq j \leq i-1} \sum_{k \in\mathbb{N}} \frac{1}{1+k^2\pi^2} \Vert R\phi_{i,k}^j \Vert_{H^1}^2
< \infty 
\end{multline} 
where we have used \eqref{eq: estimate phi_i_k^j}. Again, all the asymptotic estimates derived below hold for $k \rightarrow + \infty$, uniformly with respect to $i,j$ and $x \in [0,1]$. We know that $R{\phi}_{i,k}^j(x) = \mathrm{O}(1/k)$, hence $\Vert R{\phi}_{i,k}^j \Vert_{L^2}^2 = \mathrm{O}(1/k^2)$. Hence it suffices to analyze the asymptotic behavior of $\Vert (R{\phi}_{i,k}^j)' \Vert_{L^2}^2$. For integers $1 \leq j \leq i-1$, using \eqref{eq: simple eig phi}, we have
\begin{equation}\label{eq: inter comp Riesz basis H1}
(\phi_{i,k}^j)'(x) = \frac{(-1)^{k+i-j+1} \mu_k r_{i,k}^{j}}{\prod_{j'=j}^{i-1}r_{i,k}^{j'}\sinh(r_{i,k}^{j'})} \sinh(r_{i,k}^{j}(1-x)) .
\end{equation}
From \eqref{eq: r_i_j_k} we deduce that
\begin{multline*}
\sinh( r_{i,k}^{j} (1-x) )
= - i (-1)^k \sin ( k \pi x ) \cos ( \beta_{i,k}^{j} (1-x) ) + i (-1)^k \cos ( k \pi x ) \sin ( \beta_{i,k}^{j} (1-x) ) \\
= - i (-1)^{k} \sin(k\pi x) + \mathrm{O}(1/k)
\end{multline*}
where the last equality is inferred from \eqref{eq: cos_beta and sin_beta}. Combining the latter asymptotic behavior with \eqref{eq: r_i_j_k}, we have
$r_{i,k}^{j} \sinh( r_{i,k}^{j} (1-x) ) = (-1)^k k \pi \sin(k \pi x) + \mathrm{O}(1)$.
Using this result and \eqref{eq: estimate prod 1}, we infer from \eqref{eq: inter comp Riesz basis H1} that
\begin{equation}\label{eq: estimate diff_phi}
(\phi_{i,k}^j)'(x) = A_{i,k}^{j} \varphi_k'(x) + \mathrm{O}(1) ,
\end{equation}
hence, in view of \eqref{eq: def Rphi},
$(R{\phi}_{i,k}^j)'(x) = (\phi_{i,k}^j)'(x) - A_{i,k}^{j} \varphi_k'(x) = \mathrm{O}(1)$.
This shows that $\Vert (R{\phi}_{i,k}^j)' \Vert_{L^2}^2 = \mathrm{O}(1)$, and thus, recalling that $\Vert R{\phi}_{i,k}^j \Vert_{L^2}^2 = \mathrm{O}(1/k^2)$, we have $\Vert R{\phi}_{i,k}^j \Vert_{H^1}^2 = \mathrm{O}(1)$. Hence \eqref{eq: cond app Bari's theorem} holds true. This completes the proof.
\end{proof}

\subsection{Spectral reduction of the system}
We define the function $\varphi(x) = -\frac{1}{2}(1-x)^2$, so that $\varphi(1)=\varphi'(1)=0$ and $\varphi'(0)=1$. Making the change of variable 
\begin{equation}\label{eq: change of variable last component}
\tilde{y}^N(t,x) = y^N(t,x) - \varphi(x) u(t) ,
\end{equation}
we infer from \eqref{eq: studied PDE cascade} that
\begin{subequations}\label{eq: PDE homognenous coordinates}
\begin{align}
	& y_{t}^j = y_{xx}^j + a_j y^j , \quad 1 \leq j \leq N-1 \\
	& \tilde{y}_{t}^N = \tilde{y}_{xx}^N + a_N \tilde{y}^N + ( \varphi'' + a_N \varphi ) u - \varphi \dot{u} \\  
	& y_x^j(t,0) = \left\{ \begin{array}{ll} y^{j+1}(t,1) , & 1 \leq j \leq N-2 \\ \tilde{y}^N(t,1) , & j = N-1 \end{array} \right. \\
	& y_x^j(t,1) = 0 , \quad 1 \leq j \leq N-1 \\
	& \tilde{y}_x^N(t,0) = \tilde{y}_x^N(t,1) = 0 .
\end{align}
\end{subequations}
Defining the states $\mathcal{X} = (y^j)_{1 \leq j \leq N}$ (original coordinates) and $\tilde{\mathcal{X}} = \left( (y^j)_{1 \leq j \leq N-1} , \tilde{y}^N \right)$ (homogeneous coordinates), we infer from \eqref{eq: change of variable last component} that 
\begin{equation}\label{eq: change of variable abstract}
\tilde{\mathcal{X}} = \mathcal{X} + \beta u
\end{equation}
and \eqref{eq: PDE homognenous coordinates} is written as
\begin{subequations}\label{eq: abstract PDE}
\begin{align}
\dot{u} & = v \\
\frac{\mathrm{d}\tilde{\mathcal{X}}}{\mathrm{d}t} & = \mathcal{A} \tilde{\mathcal{X}} + \alpha u + \beta v
\end{align}
\end{subequations}
where $\alpha = (0,\ldots,0,\varphi''+a_N\varphi)$ and $\beta = (0,\ldots,0,-\varphi)$.

We now define
\begin{equation}\label{eq: coeff of projection}
x_{i,k} = \langle \mathcal{X} , \psi_{i,k} \rangle , \quad \tilde{x}_{i,k} = \langle \tilde{\mathcal{X}} , \psi_{i,k} \rangle ,
\quad
\alpha_{i,k} = \langle \alpha , \psi_{i,k} \rangle , \quad \beta_{i,k} = \langle \beta , \psi_{i,k} \rangle .
\end{equation}
By the Riesz basis property established in Lemma~\ref{eq: distinct eig - Riesz basis}, we have 
\begin{equation}\label{eq: dinstinct eig - trajectorie series expansion}
\tilde{\mathcal{X}}(t) = (y^1(t,\cdot),y^2(t,\cdot),\ldots,y^{N-1}(t,\cdot),\tilde{y}^N(t,\cdot)) 
= \sum_{i=1}^N \sum_{k \in\mathbb{N}} \tilde{x}_{i,k}(t) \phi_{i,k} . 
\end{equation}
The change of variable formula \eqref{eq: change of variable abstract} implies that
$\tilde{x}_{i,k} = x_{i,k} + \beta_{i,k} u$,
while the projection of the system dynamics in homogeneous coordinates \eqref{eq: abstract PDE} gives
\begin{equation}\label{eq: dynamics modes homogeneous coordinates}
\dot{\tilde{x}}_{i,k} = \lambda_{i,k} \tilde{x}_{i,k} + \alpha_{i,k} u + \beta_{i,k} v . 
\end{equation}
Combining the two latter equations, we infer that
\begin{equation}\label{eq: dynamics modes original coordinates}
\dot{x}_{i,k} = \lambda_{i,k} x_{i,k} + m_{i,k} u  ,
\end{equation}
where $m_{i,k} = \alpha_{i,k} + \lambda_{i,k} \beta_{i,k}$. Since all eigenvalues are simple, the controlled equation \eqref{eq: dynamics modes original coordinates} (mode $\lambda_{i,k}$) is controllable if and only if $m_{i,k} \neq 0$. The coefficients $m_{i,k}$ are characterized by the following lemma.

\begin{lem}\label{lem: obsv condition dinstinct eig}
We have
\begin{equation*}
m_{i,k} = - \psi_{i,k}^N(0) = 
\left\{ \begin{array}{ll}
-\mu_k  & \textrm{if}\quad i = N , \\
\displaystyle - \frac{(-1)^{N-i} \mu_k}{\prod_{j'=i+1}^{N}r_{i,k}^{j'}\sinh(r_{i,k}^{j'})}  & \textrm{if}\quad 1 \leq i \leq N-1 .
\end{array}\right.
\end{equation*}
In particular, $m_{i,k} \neq 0$ for all $1 \leq i \leq N$ and $k \in\mathbb{N}$.
\end{lem}

\begin{proof}
Using the definition \eqref{eq: coeff of projection} of $\alpha_{i,k}$ and using integrations by parts, we have
\begin{multline*}
\alpha_{i,k} = \int_0^1 (\varphi''+a_n\varphi) \psi_{i,k}^N \,\mathrm{d}x
= \left[ \varphi' \psi_{i,k}^N - \varphi ( \psi_{i,k}^N )' \right]_0^1 + \int_0^1 \varphi ( (\psi_{i,k}^N)'' + a_n \psi_{i,k}^N ) \,\mathrm{d}x \\
= - \psi_{i,k}^N(0) - \langle \beta , \mathcal{A}^* \psi_{i,k} \rangle 
= - \psi_{i,k}^N(0) - \lambda_{i,k} \beta_{i,k} ,
\end{multline*}
where we have used that $\varphi(1)=\varphi'(1)=0$, $\varphi'(0)=1$, and $(\psi_{i,k}^N)'(0)=0$.
\end{proof}

\begin{cor}
Under Assumption~\ref{asum: simple eigenvalues}, each mode $\lambda_{i,k}$ of the PDE cascade \eqref{eq: studied PDE cascade} is controllable.
\end{cor}

\subsection{Output feedback control strategy} 
Given $c \in L^2(0,1)$, we consider the following distributed output, measured on the last component $y^1$ of the PDE cascade \eqref{eq: studied PDE cascade}:
\begin{equation}\label{eq: bounded measurement}
z(t) = \int_0^1 c(x) y^1(t,x) \,\mathrm{d}x
\end{equation}
Our objective is to design an output feedback control strategy for \eqref{eq: studied PDE cascade} based on the measurement $z(t)$.

Note that the measurement must be done on the last component $y^1$ of the cascade; indeed a measurement done on an intermediate component of the cascade (say $y_j$ for some $2 \leq i \leq N$) would not allow to reconstruct the dynamics of the last components of the cascade (say $y_{j'}$ for $1 \leq j' \leq j-1$): observability would fail. 

Let $\delta > 0$ be the targeted exponential decay rate for the closed-loop system. Using \eqref{eq: lambda_i_k}, we choose, for any $1 \leq i \leq N$, a sufficiently large integer $n_{0,i} \in\mathbb{N}$ such that $\lambda_{i,n_{0,i}} < - \delta$, implying that $\lambda_{i,k} \leq \lambda_{i,n_{0,i}} < - \delta$ for all integers $k \geq n_{0,i}$. We now write a finite-dimensional dynamics that captures the $n_{0,i}$ first modes, for $1 \leq i \leq N$, as follows. Defining the state vector
\begin{align*}
X_{1,i} & = \begin{bmatrix} \tilde{x}_{i,0} & \tilde{x}_{i,1} & \ldots & \tilde{x}_{i,n_{0,i}-1} \end{bmatrix}^\top , \quad 1 \leq i \leq N , \\
X_1 & = \mathrm{col}(X_{1,1},X_{1,2},\dots,X_{1,n}) ,
\end{align*}
we infer from \eqref{eq: dynamics modes homogeneous coordinates} that
\begin{equation*}
\dot{X}_{1,i} = A_{1,i} X_{1,i} + B_{1,u,i} u + B_{1,v,i} v , \quad 1 \leq i \leq N,
\end{equation*}
where
\begin{multline*}
A_{1,i} = \mathrm{diag}(\lambda_{i,0},\lambda_{i,1},\ldots,\lambda_{i,n_{0,i}-1}) , \quad
B_{1,u,i} = \begin{bmatrix} \alpha_{i,0} & \alpha_{i,1} & \ldots & \alpha_{i,n_{0,i}-1} \end{bmatrix}^\top , \\
B_{1,v,i} = \begin{bmatrix} \beta_{i,0} & \beta_{i,1} & \ldots & \beta_{i,n_{0,i}-1} \end{bmatrix}^\top ,
\end{multline*}
yielding
\begin{equation}\label{eq: finite dim dynamics first modes}
\dot{X}_{1} = A_{1} X_{1} + B_{1,u} u + B_{1,v} v
\end{equation}
where
\begin{multline*}
A_{1} = \mathrm{diag}(A_{1,1},A_{1,2},\ldots,A_{1,N}) , \quad
B_{1,u} = \mathrm{col}( B_{1,u,1} , B_{1,u,2} , \ldots , B_{1,u,N} ) , \\
B_{1,v} = \mathrm{col}( B_{1,v,1} , B_{1,v,2} , \ldots , B_{1,v,N} ) .
\end{multline*}
Recalling that $v = \dot{u}$, we define the augmented state vector
$X_{1,a} = \mathrm{col}(X_1,u)$,
yielding the dynamics
\begin{equation}\label{eq: augmented finite dim dynamics first modes}
\dot{X}_{1,a} = A_{1,a} X_{1,a} + B_{1,a} v
\end{equation}
where
\begin{equation}\label{eq: augmented matrices}
A_{1,a} = \begin{bmatrix} A_1 & B_{1,u} \\ 0 & 0 \end{bmatrix} , \quad B_{1,a} = \begin{bmatrix} B_{1,v} \\ 1 \end{bmatrix} .
\end{equation}
Under Assumption~\ref{asum: simple eigenvalues} that ensures that the eigenvalues are all simple, the Hautus test implies that the pair $(A_{1,a},B_{1,a})$ is controllable if and only if $m_{i,k} = \alpha_{i,k} + \lambda_{i,k} \beta_{i,k} \neq 0$ for all $1 \leq i \leq N$ and all $0 \leq k \leq n_{0,i}-1$. Lemma~\ref{lem: obsv condition dinstinct eig} implies the following result.

\begin{lem}
Under Assumption~\ref{asum: simple eigenvalues}, the pair $(A_{1,a},B_{1,a})$ is controllable.
\end{lem} 

\begin{rem}\label{rem: state-feedback}
At this step, it is possible to design a state-feedback control strategy, as follows.
By the above lemma, there exists a matrix $K$ such that $A_{1,a}-B_{1,a}K$ has a spectral abscissa less that $-\delta < 0$, and we set the feedback control
$v = - K X_{1,a}$.
The stability of the resulting closed closed-loop system \eqref{eq: studied PDE cascade} is established in $\mathcal{H}^0$ norm, as defined by \eqref{eq: Hilbert space H^0}, thanks to the Lyapunov functional
$$
V = X_{1,a}^\top P X_{1,a} + \gamma \sum_{i=1}^{N} \sum_{k \geq n_{0,i}} \tilde{x}_{i,k}^2 ,
$$
where $P \succ 0$ is the solution of the Lyapunov equation $(A_{1,a}-B_{1,a}K)^\top P + P (A_{1,a}-B_{1,a}K) + 2 \delta P + I = 0$ and $\gamma > 0$ is chosen small enough so that $\dot{V}+2\delta V \leq 0$. For classical solutions, the stability can also be established in $\mathcal{H}^1$ norm, as defined by \eqref{eq: def space H^1}, thanks to the Lyapunov functional
$V = X_{1,a}^\top P X_{1,a} + \gamma \sum_{i=1}^{N} \sum_{k \geq n_{0,i}} (1+k^2) \tilde{x}_{i,k}^2$.
\end{rem}

In order to define our output feedback control strategy, using \eqref{eq: dinstinct eig - trajectorie series expansion}, we infer from \eqref{eq: bounded measurement} that
\begin{equation}\label{eq: output series expansion}
z(t) = \sum_{i=1}^N \sum_{k \in\mathbb{N}} c_{i,k} \tilde{x}_{i,k}(t) = C_1 X_1(t) + \sum_{i=1}^N \sum_{k \geq n_{0,i}} c_{i,k} \tilde{x}_{i,k}(t) 
\end{equation} 
where
\begin{equation}\label{eq: coeff projection measurement bounded operator}
c_{i,k} = \int_0^1 c(x) \phi_{i,k}^1(x) \,\mathrm{d}x 
\end{equation}
and
$C_{1,i}  = \begin{bmatrix} c_{i,0} & c_{i,1} & \ldots & c_{i,n_{0,i}-1} \end{bmatrix}$
and
$C_1 = \begin{bmatrix} C_{1,1} & C_{1,2} & \ldots & C_{1,N} \end{bmatrix}$.
Since the eigenvalues are simple, the Hautus test implies that the pair $(A_1,C_1)$ is observable if and only if $c_{i,k} \neq 0$ for all $1 \leq i \leq N$ and $0 \leq k \leq n_{0,i}-1$. Using the definition \eqref{eq: coeff projection measurement bounded operator} of $c_{i,k}$ and the expression of $\phi_{i,k}$ given by \eqref{eq: simple eig phi}, we infer the following result.

\begin{lem}\label{eq: simple eig obsv cond}
Under Assumption~\ref{asum: simple eigenvalues}, the pair $(A_1,C_1)$ is observable if and only if
\begin{subequations}\label{eq: obsv cond distinct eig}
\begin{align}
& \int_0^1 c(x) \cos(k \pi x)\,\mathrm{d}x \neq 0 && \textrm{if}\quad i=1 , \\
& \int_0^1 c(x) \cosh(r_{i,k}^{1}(1-x)) \,\mathrm{d}x \neq 0 && \textrm{if}\quad i \geq 2  ,
\end{align}
\end{subequations}
for all $1 \leq i \leq N$ and $0 \leq k \leq n_{0,i}-1$. Recall that $r_{i,k}^{1}$ is one of the two square roots of $\lambda_{i,k}-a_1$.
\end{lem}

We are now in a position to define our output feedback control strategy. Let $n_{i} \geq n_{0,i}$ be integers to be chosen large enough, later on. We define the following controller dynamics (see \cite{sakawa1983feedback}):
\begin{subequations}\label{eq: simple eig output feedback control strategy}
\begin{align}
\dot{u} & = v \\
\dot{\hat{X}}_1 & = A_1 \hat{X}_1 + B_{1,u} u + B_{1,v} v
 - L \left( C_1 \hat{X}_1 + \sum_{i=1}^N \sum_{k = n_{0,i}}^{n_{i}} c_{i,k} \hat{x}_{i,k} - z \right) \\
\dot{\hat{x}}_{i,k} & = \lambda_{i,k} \hat{x}_{i,k} + \alpha_{i,k} u + \beta_{i,k} v  ,\quad 1 \leq i \leq N ,\quad n_{0,i} \leq k \leq n_{i} \\
v & = -K_x \hat{X}_1 - k_u u 
\end{align}
\end{subequations}
where, setting $n_{0} = \sum_{i=1}^N n_{0,i}$, $k_u\in\mathbb{R}$ and $K_x \in\mathbb{R}^{1\times n_{0}}$ are the feedback gains while $L \in \mathbb{R}^{n_{0}}$ is the observer gain. 


\begin{thm}\label{thm1}
Let $a_1,\ldots,a_N \in\mathbb{R}$ be such that Assumption~\ref{asum: simple eigenvalues} holds and let $c \in L^2(0,1)$. Let $\delta > 0$ be arbitrary. For any $1 \leq i \leq N$, let $n_{0,i}\in\mathbb{N}$ be such that $\lambda_{i,n_{0,i}} < - \delta$. Assuming that the observability assumption \eqref{eq: obsv cond distinct eig} holds, let $K = \begin{bmatrix} K_x & k_u \end{bmatrix}\in\mathbb{R}^{1\times(n_{0}+1)}$ and $L \in \mathbb{R}^{n_{0}}$ be such that $A_{1,a}-B_{1,a}K$ and $A_1 - L C_1$ are Hurwitz with spectral abscissa less than $-\delta$. 

Then, for large enough integers $n_{i} \geq n_{0,i}$, $1 \leq i \leq N$, there exists $C > 0$ such that, for all initial conditions $y_0^j \in L^2(0,1)$, $1 \leq j \leq N$, $u(0)=u_0\in\mathbb{R}$, $\hat{X}_1(0)\in\mathbb{R}^{n_{0}}$, and $\hat{x}_{i,k}(0)\in\mathbb{R}$, the solutions of the closed-loop system consisting of the PDE cascade \eqref{eq: studied PDE cascade}, the output \eqref{eq: bounded measurement} and the controller \eqref{eq: simple eig output feedback control strategy} satisfy
\begin{multline}\label{eq: stability estimate}
\Big\Vert (y^j(t,\cdot))_{1 \leq j \leq N} \Big\Vert_{Y} + \vert u(t) \vert + \Vert \hat{X}_1(t) \Vert + \sum_{i=1}^N \sum_{k = n_{0,i}}^{n_{i}} \vert \tilde{x}_{i,k}(t) \vert  \\
\leq C e^{-\delta t} \left( \Big\Vert (y_0^j)_{1 \leq j \leq N} \Big\Vert_{Y} + \vert u_0 \vert + \Vert \hat{X}_1(0) \Vert 
+ \sum_{i=1}^N \sum_{k = n_{0,i}}^{n_{i}} \vert \tilde{x}_{i,k}(0) \vert \right) 
\end{multline} 
for every $t\geq 0$, with $Y = \mathcal{H}^0$ defined by \eqref{eq: Hilbert space H^0}. 

If the initial condition is such that $y_0^j \in H^2(0,1)$ with $(y^j)'(0) = y^{j+1}(1)$ and $(y^j)'(1) = 0$  for all $1 \leq j \leq N-1$ and $(y^N)'(0) = u(0)$ and $(y^N)'(1) = 0$, the conclusion of the theorem holds with $Y = \mathcal{H}^1$ defined by \eqref{eq: def space H^1}.
\end{thm}

\begin{proof}
Define $e_{i,k} = \tilde{x}_{i,k} - \hat{x}_{i,k}$ and
\begin{align*}
E_{1,i} & = \begin{bmatrix} e_{i,0} & e_{i,1} & \ldots & e_{i,n_{0,i}-1} \end{bmatrix} , \quad
E_1 = \mathrm{col}( E_{1,1} , E_{1,2} , \ldots , E_{1,N} ) , \\
\hat{X}_{2,i} & = \begin{bmatrix} \hat{x}_{i,n_{0,i}} & \hat{x}_{i,n_{0,i}+1} & \ldots & \hat{x}_{i,n_{i}} \end{bmatrix} , \quad
\hat{X}_2 = \mathrm{col}( \hat{X}_{2,1} , \hat{X}_{2,2} , \ldots , \hat{X}_{2,N} ) , \\
E_{2,i} & = \begin{bmatrix} (n_{0,i}+1)^\kappa e_{i,n_{0,i}} & & \ldots & (n_{i}+1)^\kappa e_{i,n_{i}} \end{bmatrix} , \quad
E_2 = \mathrm{col}( E_{2,1} , E_{2,2} , \ldots , E_{2,N} ) , \\
\hat{X}_{1,a} & = \mathrm{col}(\hat{X}^1,u), \quad X = \mathrm{col}( \hat{X}_{1,a} , E_1 , \hat{X}_2 , E_2 ) .
\end{align*}
In this proof we set $\kappa = 0$. Different values will be set later on in the proofs of Theorems~\ref{thm2} and~\ref{thm3}. Combining \eqref{eq: dynamics modes homogeneous coordinates}, \eqref{eq: finite dim dynamics first modes}, \eqref{eq: augmented finite dim dynamics first modes}, \eqref{eq: output series expansion} and \eqref{eq: simple eig output feedback control strategy}, we obtain
\begin{equation}\label{eq: dynamics truncated model simple eig}
\dot{X} = F X + \mathcal{L}\sum_{i=1}^N \zeta_i
\end{equation}
where the measurement residues are defined by
\begin{equation}\label{eq: residues of measurement}
\zeta_i = \sum_{k \geq n_{i}+1} c_{i,k} \tilde{x}_{i,k} , \quad 1 \leq i \leq N,
\end{equation}
and 
$$
F  = \begin{bmatrix}
A_{1,a} - B_{1,a} K & L_a C_1 & 0 & L_a C_2 \\
0 & A_1-LC_1 & 0 & -L C_2 \\
\left[ 0 \; B_{2,u} \right] - B_{2,v} K & 0 & A_2 & 0 \\
0 & 0 & 0 & A_2
\end{bmatrix} , 
\quad
\mathcal{L} = \mathrm{col}( L_a , - L , 0 , 0 ) ,
$$
where
\begin{align*}
K & = \begin{bmatrix} K_x & k_u \end{bmatrix} , \quad
L_a = \mathrm{col}(L,0) , \\
A_{2,i} & = \mathrm{diag}(\lambda_{i,n_{0,i}},\lambda_{i,n_{0,i}+1},\ldots,\lambda_{i,n_{i}}) , \quad
A_2 = \mathrm{diag}(A_{2,1},A_{2,2},\ldots,A_{2,N}) , \\
B_{2,u,i} & = \begin{bmatrix} \alpha_{i,n_{0,i}} & \alpha_{i,n_{0,i}+1} & \ldots & \alpha_{i,n_{i}} \end{bmatrix}^\top , \\
B_{2,u} & = \mathrm{col}(B_{2,u,1},B_{2,u,2},\ldots,B_{2,u,N}) , \quad
B_{2,v,i} = \begin{bmatrix} \beta_{i,n_{0,i}} & \beta_{i,n_{0,i}+1} & \ldots & \beta_{i,n_{i}} \end{bmatrix}^\top , \\
B_{2,v} & = \mathrm{col}(B_{2,v,1},B_{2,v,2},\ldots,B_{2,v,N}) , \\
C_{2,i} & = \begin{bmatrix} \frac{c_{i,n_{0,i}}}{(n_{0,i}+1)^\kappa} & \frac{c_{i,n_{0,i}+1}}{(n_{0,i}+2)^\kappa} & \ldots & \frac{c_{i,n_{i}}}{(n_{i}+1)^\kappa} \end{bmatrix} , \quad
C_2 = \begin{bmatrix} C_{2,1} & C_{2,2} & \ldots & C_{2,N} \end{bmatrix} .
\end{align*}
We recall that $\kappa = 0$, which ensures that $\Vert C_{2,i} \Vert = \mathrm{O}(1)$ as $n_{i} \rightarrow + \infty$. Finally the input $u$ and its time derivative $v$ are 
\begin{equation}\label{eq: express u and v}
u = E \hat{X}_{1,a} = \tilde{E} X , \quad
v = - K \hat{X}_{1,a} = - \tilde{K} X ,
\end{equation}
where $E = \begin{bmatrix} 0 & \ldots & 0 & 1 \end{bmatrix}$, $\tilde{E} = \begin{bmatrix} E & 0 & 0 & 0 \end{bmatrix}$, and $\tilde{K} = \begin{bmatrix} K & 0 & 0 & 0 \end{bmatrix}$.

To establish the claimed stability estimate \eqref{eq: stability estimate} in $\mathcal{H}^0$ norm, we define the Lyapunov functional candidate
\begin{equation}\label{eq: output-feedback Lyap H0}
V = X^\top P X + \sum_{i=1}^{N} \sum_{k \geq n_{i}+1} \tilde{x}_{i,k}^2
\end{equation}  
for some matrix $P \succ 0$. The equivalence between \eqref{eq: output-feedback Lyap H0} and the (squared) $\mathcal{H}^0$ norm holds true by Lemma~\ref{eq: distinct eig - Riesz basis}. In case of classical solutions, establishing the stability estimate \eqref{eq: stability estimate} in $\mathcal{H}^1$ norm is obtained thanks to the Lyapunov functional candidate
\begin{equation}\label{eq: output-feedback Lyap H1}
V = X^\top P X + \sum_{i=1}^{N} \sum_{k \geq n_{i}+1} (1+k^2) \tilde{x}_{i,k}^2 
\end{equation}
for some matrix $P \succ 0$.  The equivalence between \eqref{eq: output-feedback Lyap H1} and the $\mathcal{H}^1$ norm holds true by Lemma~\ref{lem: simple eig Riesz basis H1}. Indeed, using \eqref{eq: dinstinct eig - trajectorie series expansion}, we have
\begin{align}
\tilde{\mathcal{X}}(t) & = \sum_{i=1}^N \sum_{k \geq 0} \sqrt{1+k^2 \pi^2} \tilde{x}_{i,k}(t) \frac{1}{\sqrt{1+k^2 \pi^2}}\phi_{i,k} . \label{eq: series expansion H1 norm}
\end{align}
This series converges \emph{a priori} in $\mathcal{H}^0$ norm. But, based on Lemma~\ref{eq: distinct eig - Riesz basis}, the series also converges in $\mathcal{H}^1$ norm for classical solutions because $\tilde{\mathcal{X}}(t,\cdot) \in D(\mathcal{A})$. Using again Lemma~\ref{eq: distinct eig - Riesz basis}, we infer that $\Vert \tilde{\mathcal{X}}(t,\cdot) \Vert_{\mathcal{H}^1}^2$ is equivalent to $\sum_{i=1}^{N}\sum_{k \geq 0} (1+k^2)\tilde{x}_{i,k}(t)^2$. This justifies the definition \eqref{eq: output-feedback Lyap H1} for $V$.

Establishing the stability in $\mathcal{H}^0$ or in $\mathcal{H}^1$ norm now follows a similar reasoning, and we focus on the $\mathcal{H}^1$ norm. The computation of the time derivative of $V$ defined by \eqref{eq: output-feedback Lyap H1} along the solutions of \eqref{eq: dynamics modes homogeneous coordinates} and \eqref{eq: dynamics truncated model simple eig}, combined with the use of Young's inequality, gives
	\begin{align*}
	\dot{V}+2\delta V 
	& = \tilde{X}^\top 
	\begin{bmatrix} F^\top P + P F + 2 \delta P & P\mathcal{L} & \ldots & P\mathcal{L} \\  
	\mathcal{L}^\top P & 0 & \ldots & 0 \\ 
	\vdots & \vdots & \ddots & \vdots \\
	\mathcal{L}^\top P & 0 & \ldots & 0 \end{bmatrix}	 
	\tilde{X} \\
	& \phantom{=}\; + 2 \sum_{i=1}^N \sum_{k \geq n_{i} +1 } (1+k^2) (\lambda_{i,k}+\delta) \tilde{x}_{i,k}^2 
	+ 2 \sum_{i=1}^N \sum_{k \geq n_{i} + 1} (1+k^2) \tilde{x}_{i,k} ( \alpha_{i,k} u + \beta_{i,k} v ) \\
	& \leq \tilde{X}^\top 
	\begin{bmatrix} \Theta_{1,1} & P\mathcal{L} & \ldots & P\mathcal{L} \\  
	\mathcal{L}^\top P & 0 & \ldots & 0 \\ 
	\vdots & \vdots & \ddots & \vdots \\
	\mathcal{L}^\top P & 0 & \ldots & 0 \end{bmatrix}	
	\tilde{X} 
	+ 2 \sum_{i=1}^N \sum_{k \geq n_{i} +1 } (1+k^2) \left( \lambda_{i,k} + \frac{1+k^2}{\epsilon} + \delta \right) \tilde{x}_{i,k}^2 
	\end{align*}
	for any $\epsilon > 0$, with $\tilde{X} = \mathrm{col}(X,\zeta_1,\zeta_2,\ldots,\zeta_N)$ and
	$$
	\Theta_{1,1} = F^\top P + P F + 2 \delta P 
	+ \epsilon \left( \sum_{1=1}^N S_{\alpha_i,n_{i}} \tilde{E}^\top \tilde{E} + \sum_{1=1}^N S_{\beta_i,n_{i}} \tilde{K}^\top \tilde{K} \right)
	$$
	where we have used \eqref{eq: express u and v} and we have defined $S_{\alpha_i,n_{i}} = \sum_{k \geq n_{i} +1} \alpha_{i,k}^2 < \infty$ and $S_{\beta_i,n_{i}} = \sum_{k \geq n_{i} +1} \beta_{i,k}^2 < \infty$. Using \eqref{eq: residues of measurement} and the Cauchy-Schwarz inequality, we infer that 
	\begin{equation*}
		\zeta_i^2 \leq \underbrace{\sum_{k \geq n_{i}+1} c_{i,k}^2}_{= S_{\zeta_i,n_{i}} < \infty} \sum_{k \geq n_{i}+1} \tilde{x}_{i,k}^2 .
	\end{equation*}	
	 Combining the two latter estimates, we get
	\begin{equation}\label{eq: dotV}
	\dot{V}+2\delta V 
	\leq \tilde{X}^\top \Theta_1 \tilde{X} +\sum_{i=1}^N \sum_{k \geq n_{i} +1 } (1+k^2) \Gamma_{i,k} \tilde{x}_{i,k}^2 
	\end{equation}
	where
	\begin{subequations}
	\begin{align}
	\Theta_1 & = \begin{bmatrix} \Theta_{1,1} & P\mathcal{L} & \ldots & P\mathcal{L} \\  
	\mathcal{L}^\top P & -\eta_1 & \ldots & 0 \\ 
	\vdots & \vdots & \ddots & \vdots \\
	\mathcal{L}^\top P & 0 & \ldots & -\eta_N \end{bmatrix} , \label{eq: Theta_1} \\
	\Gamma_{i,k} & = 2 \left( \lambda_{i,k} + \frac{1+k^2}{\epsilon} + \delta \right) + \frac{\eta_i S_{\zeta_i,n_{i}}}{1+k^2} 
	= 2 \left( - k^2 \left( \pi^2 - \frac{1}{\epsilon} \right) + a_i + \delta + \frac{1}{\epsilon} \right) + \frac{\eta_i S_{\zeta_i,n_{i}}}{1+k^2} \label{eq: Gamma_i_k}
	\end{align}
	\end{subequations}
	for any $\epsilon,\eta_1,\ldots,\eta_N > 0$. In particular, for all $\epsilon > 1/\pi^2$, we have $\Gamma_{i,k} \leq \Gamma_{i,n_{i}+1}$ for all $k \geq n_{i}+1$. We thus infer that
	\begin{align*}
	\dot{V}+2\delta V 
	& \leq \tilde{X}^\top \Theta_1 \tilde{X} 
	+ \sum_{1 \leq i \leq N} \Gamma_{i,n_{i}+1} \sum_{k \geq n_{i}+1} (1+k^2) \tilde{x}_{i,k}^2
	\end{align*}
	for all $\epsilon > 1/\pi^2$ and all $\eta_1,\ldots,\eta_N > 0$. Hence $\dot{V}+2\delta V \leq 0$, giving the claimed stability estimate \eqref{eq: stability estimate}, provided that there exist integers $n_{i} \geq n_{0,i} + 1$, real numbers $\epsilon > 1/\pi^2$ and $\eta_1,\ldots,\eta_N > 0$, and a matrix $P \succ 0$ such that
	\begin{equation*}
	\Theta_1 \preceq 0 , \quad
	 \Gamma_{i,n_{i}+1} \leq 0 , \quad 1 \leq i \leq N .
	\end{equation*}
	Let us prove the feasibility of these constraints. Recall that $F$ is Hurwitz with spectral abscissa less than $-\delta$. Note also that $\Vert C_2 \Vert = \mathrm{O}(1)$, $\Vert B_{2,u} \Vert = \mathrm{O}(1)$, and $\Vert B_{2,v} \Vert = \mathrm{O}(1)$ as $(n_{i})_{1 \leq i \leq N} \rightarrow + \infty$. Hence, the application of the lemma in Appendix of \cite{lhachemi2020finite} shows that the solution $P \succ 0$ to $F^\top P + P F + 2 \delta P = -I$ is such that $\Vert P \Vert = \mathrm{O}(1)$ as $(n_{i})_{1 \leq i \leq N} \rightarrow + \infty$. We now arbitrarily fix the constant $\epsilon > 1/\pi^2$ and we define
	\begin{equation*}
\eta_i = \left\{\begin{array}{cl}
\frac{1}{\sqrt{S_{\zeta_i,n_{i}}}} & \mathrm{if}\; S_{\zeta_i,n_{i}} \neq 0 \\
n_{i} & \mathrm{otherwise}
\end{array}\right.
, \quad 1 \leq i \leq N .
\end{equation*}
	We then infer that, for all $1 \leq i \leq N$, $\Gamma_{i,n_{i}+1} \rightarrow - \infty$ as $n_{i} \rightarrow +\infty$. Moreover, the Schur complement shows that  $\Theta_1 \preceq 0$ for $n_{1},n_{2},\ldots,n_{N}$ selected sufficiently large. The last claim uses the facts that $\Vert \tilde{E} \Vert = 1$, $\Vert \tilde{K} \Vert = \Vert K \Vert$, and $\Vert \mathcal{L} \Vert = \sqrt{2} \Vert L \Vert$ are constants not depending on $n_{i}$ while $\Vert P \Vert = \mathrm{O}(1)$ as $(n_{i})_{1 \leq i \leq N} \rightarrow + \infty$ and that $S_{\alpha_i,n_{i}},S_{\beta_i,n_{i}} \rightarrow 0$ as $n_{i} \rightarrow +\infty$. The proof is complete.
\end{proof}

\subsection{Extension to pointwise measurements}
\subsubsection{Pointwise measurement of $y^1$} 
In this section, we briefly show how to adapt our previous result to the case of a pointwise measurement of $y^1$ of the form:
\begin{equation}\label{eq: Dirichlet measurement}
	z(t) = y^1(t,\xi_p) 
\end{equation}
for some $\xi_p\in[0,1]$. Using \eqref{eq: dinstinct eig - trajectorie series expansion}, the series expansion \eqref{eq: output series expansion} still holds true but with $c_{i,k} = \phi_{i,k}^1(\xi_p)$. Similarly to Lemma~\ref{eq: simple eig obsv cond}, and in view of the expression of $\phi_{i,k}^1$ provided by \eqref{eq: simple eig phi}, the pair $(A_1,C_1)$ is observable if and only if:
\begin{subequations}\label{eq: obsv cond for Dirichlet}
\begin{align}
& \cos(k\pi\xi_p) \neq 0 , \quad 1 \leq k \leq N_{0,1}-1 , \\
& \cosh(r_{i,k}^{1}(1-\xi_p)) \neq 0 , \quad 2 \leq i \leq N ,\; 0 \leq k \leq n_{0,i}-1 ,
\end{align}
\end{subequations}
where $r_{i,k}^{1}$ is one of the two square roots of $\lambda_{i,k}-a_1$. 

\begin{thm}\label{thm2}
Considering the pointwise measurement \eqref{eq: Dirichlet measurement} for some $\xi_p \in [0,1]$, replacing the observability condition \eqref{eq: obsv cond distinct eig} by \eqref{eq: obsv cond for Dirichlet}, the same statement as Theorem~\ref{thm1} holds true for the Hilbert space $Y = \mathcal{H}^0$ defined by \eqref{eq: Hilbert space H^0}.
\end{thm} 

\begin{proof}
The proof follows the one of Theorem~\ref{thm1} except for the estimation of the measurement residues $\zeta_i$ defined by \eqref{eq: residues of measurement}. Recall that for the pointwise measurement \eqref{eq: Dirichlet measurement} we have defined $c_{i,k} = \phi_{i,k}^1(\xi_p)$. In view of \eqref{eq: simple eig phi} in the case $i = 1$ and \eqref{eq: estimate phi_i_k^j} in the case $2 \leq i \leq N$, we have $c_{i,k}= \mathrm{O}(1)$ as $k \rightarrow +\infty$. Consequently we set $\kappa = 1$ so that we have $\Vert C_{2,i} \Vert = \mathrm{O}(1)$ as $n_{i} \rightarrow + \infty$, hence $\Vert C_{2} \Vert = \mathrm{O}(1)$ as $(n_{i})_{1\leq i \leq N} \rightarrow + \infty$. Moreover we have 
\begin{align*}
	\zeta_i^2 \leq \underbrace{\sum_{k \geq n_{i}+1} \frac{c_{i,k}^2}{1+k^2}}_{= S_{\zeta_i,n_{i}} < \infty} \sum_{k \geq n_{i}+1} (1+k^2) \tilde{x}_{i,k}^2 .
\end{align*}	
We thus infer that \eqref{eq: dotV} holds with
$$
	\Gamma_{i,k} = 2 \left( \lambda_{i,k} + \frac{1+k^2}{\epsilon} + \delta \right) + \eta_i S_{\zeta_i,n_{i}} 
	= 2 \left( - k^2 \left( \pi^2 - \frac{1}{\epsilon} \right) + a_i + \delta + \frac{1}{\epsilon} \right) + \eta_i S_{\zeta_i,n_{i}} 
$$
while $\Theta_1$ is still given by \eqref{eq: Theta_1}. The end of the proof follows the arguments of the proof of Theorem~\ref{thm1}. 
\end{proof}

\subsubsection{Pointwise measurement of $y^1_x$}\label{subsubsec: Neumann measurement} 

The case of a pointwise measurement of $y^1_x$ of the form:
\begin{equation}\label{eq: Neumann measurement}
	z(t) = y^1_x(t,\xi_p) ,
\end{equation}
for some $\xi_p\in[0,1]$, is more stringent. This is due to the fact that the system consisting of the PDE cascade  \eqref{eq: studied PDE cascade} and the output \eqref{eq: Neumann measurement} is not observable. Indeed, consider the initial condition $y_0^1(x) = 1$ and $y_0^i(x) = 0$ for $2 \leq i \leq N$. We infer that the solution of \eqref{eq: studied PDE cascade} is $y^1(t,x) = e^{a_1 t}$ and $y^i(t,x) = 0$ for $2 \leq i \leq N$. In particular, \eqref{eq: Neumann measurement} gives an identically zero measurement $z=0$ for a nontrivial state trajectory. This remark highlights the fact that if $a_1 \geq 0$, which corresponds to the case of a cascade \eqref{eq: studied PDE cascade} that is not open-loop exponentially stable since $\lambda_{1,0} = a_1 \geq 0$, then the pointwise measurement of $y_x$ cannot be used to design an output feedback control strategy  for stabilizing the PDE cascade \eqref{eq: studied PDE cascade}. 

However, the approach developed in this paper still presents some merits in the case of the pointwise measurement \eqref{eq: Neumann measurement} if we assume moreover that the last component of the PDE cascade \eqref{eq: studied PDE cascade}, described by $y^1$, is open-loop exponentially stable. Indeed, assume that $a_1 < 0$ and let us choose the targeted exponential decay rate for the closed-loop system trajectories as $\delta \in (0, -a_1)$. Then, following the construction reported in the previous subsection, no mode $\lambda_{1,k}$ has to be included in the construction of the matrices to buid the finite-dimensional dynamics \eqref{eq: finite dim dynamics first modes}. This is achieved by setting $N_{0,1} = 0$. Then, using \eqref{eq: dinstinct eig - trajectorie series expansion}, the series expansion \eqref{eq: output series expansion} still holds true but with $c_{i,k} = (\phi_{i,k}^1)'(\xi_p)$. Now, similarly to Lemma~\ref{eq: simple eig obsv cond} and in view of the expression of $\phi_{i,k}^1$ provided by \eqref{eq: simple eig phi}, the pair $(A_1,C_1)$ is observable if and only if:
\begin{equation}\label{eq: obsv cond for Neumann}
\sinh(r_{i,k}^{1}(1-\xi_p)) \neq 0 , \quad 2 \leq i \leq N ,\quad 0 \leq k \leq n_{0,i}-1 ,
\end{equation}
where $r_{i,k}^{1}$ is one of the two square roots of $\lambda_{i,k}-a_1$. Note that Assumption~\ref{asum: simple eigenvalues} implies that $r_{i,k}^{1} \neq 0$ for $i \geq 2$. 

\begin{rem}
The fact that the mode $\lambda_{1,0} = a_1$ is never observable for the pointwise measurement \eqref{eq: Neumann measurement} is reflected by the fact that $\phi_{1,0}^1(x) = 1$, hence $c_{1,0} = (\phi_{1,0}^1)'(\xi_p) = 0$.
\end{rem}


\begin{thm}\label{thm3}
In the setting of Theorem~\ref{thm2}, we further assume that $a_1 < 0$ and we choose $\delta \in (0,-a_1)$, we set $N_{0,1} = 0$, the observability assumption \eqref{eq: obsv cond for Dirichlet} is replaced by \eqref{eq: obsv cond for Neumann}, the system output \eqref{eq: Dirichlet measurement} is replaced by the Neumann measurement \eqref{eq: Neumann measurement}.
Then the statement of Theorem~\ref{thm2} holds true.
\end{thm} 

\begin{rem}
The constraint $\delta \in (0,-a_1)$ reflects the fact that, in the context of an output feedback, the exponential stability of the closed-loop system, in the sense of \eqref{eq: stability estimate}, is associated with an exponential decay rate $\delta$ that cannot be faster than the slowest unobservable mode $\lambda_{1,0} = a_1 < 0$.
\end{rem}

\begin{proof}
The only difference with the proof of Theorem~\ref{thm2} is in the estimation of the measurement residues $\zeta_i$ defined by \eqref{eq: residues of measurement}. For the pointwise measurement \eqref{eq: Neumann measurement} we have defined $c_{i,k} = (\phi_{i,k}^1)'(\xi_p)$. Based on \eqref{eq: simple eig phi} in the case $i = 1$ and \eqref{eq: estimate diff_phi} in the case $2 \leq i \leq N$, we note that $c_{i,k}= \mathrm{O}(1)$ as $k \rightarrow +\infty$. Thus, we set $\kappa = 7/4$ so that we have $\Vert C_{2,i} \Vert = \mathrm{O}(1)$ as $n_{i} \rightarrow + \infty$, hence $\Vert C_{2} \Vert = \mathrm{O}(1)$ as $(n_{i})_{1\leq i \leq N} \rightarrow + \infty$. We also have 
\begin{align*}
	\zeta_i^2 \leq \underbrace{\sum_{k \geq n_{i}+1} \frac{c_{i,k}^2}{(1+k^2)^{7/4}}}_{= S_{\zeta_i,n_{i}} < \infty} \sum_{k \geq n_{i}+1} (1+k^2)^{7/4} \tilde{x}_{i,k}^2 .
\end{align*}	
Therefore, we infer that \eqref{eq: dotV} holds with
\begin{multline*}
	\Gamma_{i,k} = 2 \left( \lambda_{i,k} + \frac{1+k^2}{\epsilon} + \delta \right) + (1+k^2)^{3/4} \eta_i S_{\zeta_i,n_{i}}  \\
	= 2 \left( - k^2 \left( \pi^2 - \frac{1}{\epsilon} \right) + a_i + \delta + \frac{1}{\epsilon} \right) 
	+ (1+k^2)^{3/4} \eta_i S_{\zeta_i,n_{i}} 
\end{multline*}
while $\Theta_1$ is still given by \eqref{eq: Theta_1}. The proof then follows the same arguments as before.
\end{proof}

\section{Cascade of $N$ identical heat equations}\label{sec: mult eig}
In this section, we address the case of $N$ identical heat equations and we make the following standing assumption:

\begin{assum}\label{assump: identical heat equations}
$\quad a \triangleq a_1 = a_2 = \ldots = a_N$.
\end{assum}

\subsection{Spectral properties}
In the Hilbert space $\mathcal{H}^0$ defined by \eqref{eq: Hilbert space H^0}, we consider the unbounded operator
$\mathcal{A} ( (f^j)_{1 \leq j \leq N} ) = ((f^j)''+a f^j)_{1 \leq j \leq N}$
on the domain defined by \eqref{eq: 2b2dev operator A - domain}. We start by studying the generalized eigenelements of $\mathcal{A}$. 

\begin{lem}\label{lemma: prel for identical heat equations}
Let $k \in\mathbb{N}$, $\gamma\in\mathbb{R}$, and $g_0,g_1 \in L^2(0,1)$ be such that $\int_0^1 \cos(k\pi s) g_1(s) \,\mathrm{d}s \neq 0$. Then, the functions $f:[0,1]\rightarrow\mathbb{R}$ and real numbers $\nu\in\mathbb{R}$ that satisfy
$f'' + k^2 \pi^2 f = g_0 + \nu g_1$, $f'(0) = \gamma$, $f'(1)=0$,
are given by
\begin{equation}\label{eq:ch:june}
\nu = - \frac{\gamma + \int_0^1 g_0(s) \,\mathrm{d}s}{\int_0^1 g_1(s) \,\mathrm{d}s},
\qquad
f(x) = c + \gamma x + \int_0^x (x-s) ( g_0(s) + \nu g_1(s) ) \,\mathrm{d}s,
\end{equation}
if $k = 0$, and by
\begin{equation*}
\begin{split}
\nu &= - \frac{\gamma + \int_0^1 \cos(k\pi s) g_0(s) \,\mathrm{d}s}{\int_0^1 \cos(k\pi s) g_1(s) \,\mathrm{d}s}  ,
\\
f(x) &= c \sqrt{2} \cos(k \pi x) + \frac{\gamma}{k \pi} \sin(k \pi x) 
+ \frac{1}{k\pi} \int_0^x \sin(k \pi (x-s) ) ( g_0(s) + \nu g_1(s) ) \,\mathrm{d}s,
\end{split}
\end{equation*}
if $k \geq 1$, where $c \in\mathbb{R}$ is an arbitrary constant that corresponds to the solutions to the homogeneous ODE (\ref{eq:ch:june}), i.e., when $g_0 = g_1 = 0$ and $\gamma = 0$. 
\end{lem}


\begin{lem}\label{eq: integration in series trough heat equations}
We define a sequence of functions as follows. For $k = 0$, we first set 
\begin{subequations}\label{eq: def sigma_0}
\begin{equation}\label{eq: def sigma_0^1}
	\sigma_0^{1}(x) = 1 .
\end{equation}
Then, we iteratively define, for integers $n \geq 2$,
\begin{equation}\label{eq: def sigma_0^n}
	\sigma_0^{n}(x) = \sigma_0^{n-1}(1) x + \sum_{i=1}^{n-1} \nu_0^i \int_0^x (x-s) \sigma_{0}^{n-i}(s) \,\mathrm{d}s .
\end{equation}
where
\begin{equation}
\nu_0^{n-1} = - \sigma_0^{n-1}(1) - \sum_{i=1}^{n-2} \nu_0^i \int_0^1 \sigma_0^{n-i}(s) \,\mathrm{d}s  .
\end{equation}
\end{subequations}
For positive integers $k \geq 1$, we first set 
\begin{subequations}\label{eq: def sigma_k}
\begin{equation}\label{eq: def sigma_k^1}
	\sigma_k^{1}(x) = \sqrt{2} \cos(k\pi x) .
\end{equation}
Then, we iteratively define, for integers $n \geq 2$,
\begin{equation}
	\sigma_k^{n}(x) = \frac{\sigma_k^{n-1}(1)}{k\pi} \sin(k\pi x) 
	+ \frac{1}{k\pi} \sum_{i=1}^{n-1} \nu_k^i \int_0^x \sin(k\pi(x-s)) \sigma_{k}^{n-i}(s) \,\mathrm{d}s . \label{eq: def sigma_k^n}
\end{equation}
where
\begin{equation}
\nu_k^{n-1} = -\sqrt{2} \left( \sigma_k^{n-1}(1) + \sum_{i=1}^{n-2} \nu_k^i \int_0^1 \cos(k\pi s) \sigma_k^{n-i}(s) \,\mathrm{d}s \right)  .
\end{equation}
\end{subequations}
Then we have, for $n=1$ and for all $k \in \mathbb{N}$,
\begin{subequations}
\begin{equation}\label{eq: edo sigma^1}
(\sigma_k^1)''+k^2\pi^2 \sigma_k^1 = 0 , \quad (\sigma_k^1)'(0)=(\sigma_k^1)'(1)=0 ,
\end{equation}
while, for all integers $n \geq 2$ and for all $k \in \mathbb{N}$,
\begin{align}\label{eq: edo sigma^k}
& (\sigma_k^n)''+k^2\pi^2 \sigma_k^n = \sum_{i=1}^{n-1} \nu_k^i \sigma_k^{n-i} , \\ 
& (\sigma_k^n)'(0)= \sigma_k^{n-1}(1) ,\quad (\sigma_k^n)'(1)=0 . \label{eq: edo sigma^k - 2}
\end{align}
\end{subequations}
In particular, 
\begin{subequations}\label{eq: asymptotic estimate sigma}
\begin{align}
& \Vert \sigma_k^n \Vert_{L^\infty} = \mathrm{O}(1) 
\quad\textrm{and}\quad
\nu_k^n = \mathrm{O}(1) && \textrm{if}\quad n \geq 1, \label{eq: asymptotic estimate sigma - 12} \\
& \Vert \sigma_k^n \Vert_{L^\infty} = \mathrm{O}(1/k) && \textrm{if}\quad n \geq 2 , \label{eq: asymptotic estimate sigma - 3}
\end{align}
\end{subequations}
as $k \rightarrow + \infty$. 
Finally, 
\begin{equation}\label{eq: sigmak polnomial * trig}
\sigma_k^n(x) = P_k^n(x) \cos(k \pi x) + Q_k^n(x) \sin(k \pi x).
\end{equation}
for some polynomials $P_k^n,Q_k^n \in\mathbb{R}[X]$.
\end{lem}

\begin{proof}
First, \eqref{eq: edo sigma^1} follows from the definitions \eqref{eq: def sigma_0^1} and \eqref{eq: def sigma_k^1} of $\sigma_k^1$. The proof now consists, for $n \geq 2$, of applying iteratively Lemma~\ref{lemma: prel for identical heat equations} by setting $g_0(x) = \sum_{i=1}^{n-2} \nu_k^i \sigma_k^{n-i}$ and $g_1(x) = \sigma_k^{1}$ while taking $c = 0$. It has only to be checked that $\int_0^1 \cos(k\pi s) g_1(s) \,\mathrm{d}s \neq 0$. This is satisfied because for $k = 0$, due to \eqref{eq: def sigma_0^1}, $\int_0^1 \sigma_0^1(s) \,\mathrm{d}s = 1$ and, for $k \geq 1$, due to \eqref{eq: def sigma_k^1}, $\int_0^1 \cos(k \pi s) \sigma_k^1(s) \,\mathrm{d}s = \sqrt{2}/2$. Finally, \eqref{eq: asymptotic estimate sigma} and \eqref{eq: sigmak polnomial * trig} are obtained by an induction argument based on \eqref{eq: def sigma_0} and \eqref{eq: def sigma_k}. This completes the proof. 
\end{proof}

\begin{rem}\label{rem: computation firsts sigmak}
The functions $\sigma_k^n$ can be iteratively computed using \eqref{eq: def sigma_0} and \eqref{eq: def sigma_k}, adequately tuning the quantities $\nu_k^n$ so as to satisfy the first boundary condition of \eqref{eq: edo sigma^k - 2}. For a given number $N$ of heat equations in the PDE cascade \eqref{eq: studied PDE cascade}, we will see next that only the computation of $\sigma_k^n$ for $1 \leq n \leq N$ is required. The first four computations give 
\begin{align*}
\sigma_0^{1}(x) & = 1 , \\
\sigma_0^{2}(x) & = - \frac{x^2}{2} + x , \\
\sigma_0^{3}(x) & = \frac{x^4}{24} - \frac{x^3}{6} - \frac{x^2}{12} + \frac{x}{2} , \\
\sigma_0^{4}(x) & = - \frac{x^6}{720} + \frac{x^5}{120} + \frac{x^4}{72} - \frac{x^3}{9} - \frac{17x^2}{720} + \frac{7x}{24} ,
\end{align*}
with 
$\nu_0^1 = -1$,
$\nu_0^2 = -1/6$,
$\nu_0^3 = -17/360$,
$\nu_0^4 = -31/1890$,
and, for $k \geq 1$,
\begin{align*}
\sigma_k^{1}(x) & = \sqrt{2} \cos(k\pi x) , \\
\sigma_k^{2}(x) & = \frac{(-1)^k \sqrt{2}}{k\pi}(1-x) \sin(k\pi x) , \\
\sigma_k^{3}(x) & = \frac{\sqrt{2}}{k^3\pi^3}(x-1) \sin(k\pi x) + \frac{\sqrt{2}}{2k^2\pi^2}x(2-x) \cos(k\pi x) , \\
\sigma_k^{4}(x) & =  (-1)^k \sqrt{2}  \left( \frac{3-x-3x^2+x^3}{6k^3\pi^3} + \frac{2-2x}{k^5\pi^5} \right) \sin(k \pi x) + (-1)^k\sqrt{2} \frac{x^2-2x}{k^4 \pi^4} \cos(k\pi x)
\end{align*}
with 
$\nu_k^1 = 2(-1)^{k+1}$,
$\nu_k^2 = \frac{1}{k^2\pi^2}$,
$\nu_k^3 = (-1)^{k+1}\frac{k^2\pi^2 + 6}{3k^4\pi^4}$,
$\nu_k^4 = \frac{k^2 \pi^2 + 5}{k^6 \pi^6}$.
\end{rem}


\begin{lem}\label{lem: mult eig - phi}
Under Assumption~\ref{assump: identical heat equations}, the eigenvalues of $\mathcal{A}$ are given by
\begin{equation}\label{eq: lambda_k}
\lambda_{k} = a - k^2 \pi^2 ,\quad k \geq 0 .
\end{equation}
They are of geometric multiplicity $1$ and of algebraic multiplicity $N$ with associated generalized eigenvectors
\begin{equation}
\phi_{i,k} = ( \phi_{i,k}^j )_{1 \leq j \leq N} 
= ( \underbrace{\sigma_k^i , \sigma_k^{i-1} , \ldots , \sigma_k^1}_{i\,\mathrm{elements}} , \underbrace{0 , \ldots , 0}_{N-i \,\mathrm{times}} ) \in\mathcal{D}(A) \label{eq: mult eig - def phi}
\end{equation}
for integers $1 \leq i \leq N$ and $k \in\mathbb{N}$, that satisfy
\begin{equation}\label{eq: expression A for gen eig vec}
\mathcal{A} \phi_{i,k}
= \lambda_{k} \phi_{i,k} + \sum_{j=1}^{i-1} \nu_k^{i-j} \phi_{j,k} .
\end{equation}
Hence, in $(\phi_{i,k})_{1 \leq i \leq N}$, the matrix of (the restriction of) $\mathcal{A}$ is given by 
\begin{equation}\label{eq: matrix Mk}
M_k = \begin{bmatrix}
\lambda_k & \nu_k^1 & \nu_k^2 & \nu_k^3 & \ldots & \nu_k^{N-1} \\
0 & \lambda_k & \nu_k^1 & \nu_k^2 & \ldots & \nu_k^{N-2} \\
0 & 0 & \lambda_k & \nu_k^1 & \ddots & \vdots \\
0 & 0 & 0 & \lambda_k & \ddots & \nu_k^2 \\
\vdots & \vdots & \ddots & \ddots & \lambda_k & \nu_k^1 \\
0 & 0 & \ldots & 0 & 0 & \lambda_k \\ 
\end{bmatrix} 
\end{equation}
which has $\lambda_k$ as unique eigenvalue, of geometric multiplicity $1$ and of algebraic multiplicity $N$.
\end{lem}

\begin{rem}
Lemma~\ref{eq: integration in series trough heat equations} is derived by iteratively applying Lemma~\ref{lemma: prel for identical heat equations} with $c = 0$ at each step. This choice is arbitrary: a different value of $c$ could be chosen at each step of the iteration. This would change the eigenelements $\phi_{i,k}$ defined in Lemma~\ref{lem: mult eig - phi}. This reflects the fact that there is one degree of freedom\footnote{The number of degrees of freedom coincides with the geometric multiplicity of the associated eigenvalues, which is $1$ in our case.} in the definition of each of the generalized eigenvectors $\phi_{i,k}$ for $2 \leq i \leq N$. This degree of freedom corresponds to the fact that a generalized eigenvector translated by an eigenvector associated with the same eigenvalue remains a generalized eigenvector. Since constants $c \neq 0$ would make the expressions of $\phi_{i,k}$ more complex, there is no technical benefit of choosing $c\neq 0$. Note however that this choice fixes the corresponding constants in the dual construction (see Remark~\ref{rem: dual constants fixed for biorthogonality}).
\end{rem}

\begin{rem}
Since $M_k$ has one single eigenvalue, of geometric multiplicity $1$, $M_k$ is similar to a Jordan block.
\end{rem}


\begin{lem}\label{eq: mult eig - Riesz basis}
$\Phi = \{ \phi_{i,k} \,\mid\, 1 \leq i \leq N ,\; k \in\mathbb{N} \}$ is a Riesz basis of $\mathcal{H}^0$. Moreover, $\mathcal{A}$ generates a $C_0$-semigroup.
\end{lem}

\begin{proof}
Define $\hat{\Phi} = \{ \hat{\phi}_{i,k} \,\mid\, 1 \leq i \leq N ,\; k \in\mathbb{N} \}$ with $\hat{\phi}_{i,k} = ( \hat{\phi}_{i,k}^j )_{1 \leq j \leq N}$ where $\hat{\phi}_{i,k}^j(x) = \delta_{i,j} \varphi_k(x)$ and $\varphi_k$ is defined as in \eqref{eq: simple eig phi} for $j=i$. Here $\delta_{i,i} = 1$ while $\delta_{i,j} = 0$ whenever $i \neq j$. It is easy to see that $\hat{\Phi}$ is a Hilbert basis of $\mathcal{H}^0$. Hence, noting that $\sigma_k^1 = \varphi_k$ and invoking \eqref{eq: asymptotic estimate sigma - 3}, we infer that
$$
\sum_{i=1}^N \sum_{k \geq 0} \Vert \phi_{i,k} - \hat{\phi}_{i,k} \Vert_{\mathcal{H}^0}^2
= \sum_{i=1}^N \sum_{j=1}^{i-1} \sum_{k \geq 0} \Vert \phi_{i,k}^j \Vert_{L^2}^2 
= \sum_{i=1}^N \sum_{j=2}^{i} \sum_{k \geq 0} \Vert \sigma_{k}^j \Vert_{L^2}^2 < \infty .
$$
Since $\Phi$ is $\omega$-linearly independent, Bari's theorem implies that $\Phi$ is a Riesz basis of $\mathcal{H}^0$. Finally, a similar approach to \cite[Lem.~2.1]{leiva2003lemma} shows that $\mathcal{A}$ generates a $C_0$-semigroup because $\Vert e^{M_k t} \Vert \leq g(t)$ for all $k \in\mathbb{N}$ and for some real-valued continuous function $g$ (see the beginning of the proof of Theorem~\ref{thm4}).
\end{proof}

The previous lemma provides a suitable framework for studying the solutions in $L^2$ norm. The tool for studying the solutions in $H^1$ norm is provided by the following lemma.

\begin{lem}\label{lem: mult eig Riesz basis H1}
$\tilde{\Phi} = \{ \tilde{\phi}_{i,k} \,\mid\, 1 \leq i \leq N ,\; k \in\mathbb{N} \}$, where $\tilde{\phi}_{i,k} = \frac{1}{\sqrt{1+k^2\pi^2}} \phi_{i,k}$, is a Riesz basis of $\mathcal{H}^1$ (defined by \eqref{eq: def space H^1}).
\end{lem}

\begin{proof}
Define $\check{\Phi} = \{ \check{\phi}_{i,k} \,\mid\, 1 \leq i \leq N ,\; k \in\mathbb{N} \}$ with $\check{\phi}_{i,k} = ( \check{\phi}_{i,k}^j )_{1 \leq j \leq N}$ where $\check{\phi}_{i,k}^j(x) = \frac{1}{\sqrt{1+k^2 \pi^2}} \delta_{i,j} \varphi_k(x)$. We know that $\check{\Phi}$ is a Riesz basis of $\mathcal{H}^1$. We apply Bari's theorem by comparing $\tilde{\Phi}$ to the Riesz basis $\check{\Phi}$. Since $\tilde{\Phi}$ is $\omega$-linearly independent, it remains to show that
\begin{equation}
\sum_{1 \leq i \leq N} \sum_{k \geq 0} \Vert \tilde{\phi}_{i,k} - \check{\phi}_{i,k} \Vert_{\mathcal{H}^1}^2 
= \sum_{i=1}^N \sum_{j=1}^{i-1} \sum_{k \geq 0} \frac{1}{1+k^2 \pi^2} \Vert \phi_{i,k}^j \Vert_{H^1}^2 
= \sum_{i=1}^N \sum_{j=2}^{i} \sum_{k \geq 0} \frac{1}{1+k^2 \pi^2} \Vert \sigma_{k}^j \Vert_{H^1}^2 < \infty . \label{eq: cond app Bari's theorem bis}
\end{equation} 
Based on \eqref{eq: asymptotic estimate sigma - 3}, we already know that, for any fixed $n \geq 2$, $\Vert \sigma_k^n \Vert_{L^2} = \mathrm{O}(1/k)$ as $k \rightarrow + \infty$. It remains to study the terms $\Vert (\sigma_k^n)' \Vert_{L^2}$ for $n \geq 2$ as $k \rightarrow + \infty$. We infer from \eqref{eq: def sigma_k^n} that
$$
	(\sigma_k^{n})'(x) = \sigma_k^{n-1}(1) \cos(k\pi x) 
	 + \sum_{i=1}^{n-1} \nu_k^i \int_0^x \cos(k\pi(x-s)) \sigma_{k}^{n-i}(s) \,\mathrm{d}s .
$$
Based on \eqref{eq: asymptotic estimate sigma - 12}, we thus infer that $\Vert (\sigma_k^n)' \Vert_{L^\infty} = \mathrm{O}(1)$ as $k \rightarrow + \infty$ for any given $n \geq 2$. Hence \eqref{eq: cond app Bari's theorem bis} holds true, which concludes the proof. 
\end{proof}

In order to compute the dual Riesz basis of $\Phi$, we introduce the adjoint operator 
$\mathcal{A}^* ( (g^j)_{1 \leq j \leq N} ) = ((g^j)''+a g^j)_{1 \leq j \leq N}$
defined on the domain \eqref{eq: domain A*}. 

\begin{lem}\label{lemma: prel for identical heat equations - dual}
Let $k \in\mathbb{N}$, $\gamma\in\mathbb{R}$, and $g_0,g_1 \in L^2(0,1)$ be such that $\int_0^1 \cos(k\pi s) g_1(s) \,\mathrm{d}s \neq 0$. Then, the functions $f:[0,1]\rightarrow\mathbb{R}$ and real numbers $\nu\in\mathbb{R}$ that satisfy
$f'' + k^2 \pi^2 f = g_0 + \nu g_1$, $f'(0) = 0$, $f'(1)=\gamma$,
are given by
\begin{equation*}
\nu = \frac{\gamma - \int_0^1 g_0(s) \,\mathrm{d}s}{\int_0^1 g_1(s) \,\mathrm{d}s}, 
\qquad
f(x) = c + \int_0^x (x-s) ( g_0(s) + \nu g_1(s) ) \,\mathrm{d}s,
\end{equation*}
if $k = 0$, and by
\begin{equation*}
\begin{split}
\nu &= \frac{ (-1)^k \gamma - \int_0^1 \cos(k\pi s) g_0(s) \,\mathrm{d}s}{\int_0^1 \cos(k\pi s) g_1(s) \,\mathrm{d}s} ,
\\
f(x) &= c \sqrt{2} \cos(k \pi x) 
+ \frac{1}{k\pi} \int_0^x \sin(k \pi (x-s) ) ( g_0(s) + \nu g_1(s) ) \,\mathrm{d}s,
\end{split}
\end{equation*}
if $k \geq 1$, where $c \in\mathbb{R}$ is an arbitrary constant. 
\end{lem}


\begin{lem}\label{eq: integration in series trough heat equations dual}
We define a sequence of functions as follows. For $k = 0$, we first set 
\begin{subequations}\label{eq: def sigma_0 dual}
\begin{equation}\label{eq: def sigma_0^1 dual}
	\tau_0^{1}(x) = 1 .
\end{equation}
Then, we iteratively define for integers $n \geq 2$
\begin{equation}\label{eq: def sigma_0^n dual}
	\tau_0^{n}(x) = c_0^n + \underbrace{\sum_{i=1}^{n-1} \tilde{\nu}_0^i \int_0^x (x-s) \tau_{0}^{n-i}(s) \,\mathrm{d}s}_{= \tau_{0,0}^n(x)}  .
\end{equation}
where
\begin{equation}
\tilde{\nu}_0^{n-1} = - \tau_0^{n-1}(0) - \sum_{i=1}^{n-2} \tilde{\nu}_0^i \int_0^1 \tau_0^{n-i}(s) \,\mathrm{d}s ,
\end{equation}
and
\begin{equation}\label{eq: edo sigma^1 dual constant c0n}
c_0^n = - \sum_{i=1}^{n-1} \int_0^1 \sigma_0^{n-i+1}(s) \tau_0^i(s) \,\mathrm{d}s - \int_0^1 \sigma_0^{1}(s) \tau_{0,0}^n(s) \,\mathrm{d}s . 
\end{equation}
\end{subequations}
For integers $k \geq 1$, we first set 
\begin{subequations}\label{eq: def sigma_k dual}
\begin{equation}\label{eq: def sigma_k^1 dual}
	\tau_k^{1}(x) = \sqrt{2} \cos(k\pi x) .
\end{equation}
Then, we iteratively define for integers $n \geq 2$
\begin{equation}
	\tau_k^{n}(x) = c_k^n \sqrt{2} \cos(k\pi x)  
	+ \underbrace{\frac{1}{k\pi} \sum_{i=1}^{n-1} \tilde{\nu}_k^i \int_0^x \sin(k\pi(x-s)) \tau_{k}^{n-i}(s) \,\mathrm{d}s}_{= \tau_{k,0}^n(x)}  \label{eq: def sigma_k^n dual}
\end{equation}
where
\begin{equation}
\tilde{\nu}_k^{n-1} = -\sqrt{2} \bigg( (-1)^k \tau_k^{n-1}(0) 
+ \sum_{i=1}^{n-2} \tilde{\nu}_k^i \int_0^1 \cos(k\pi s) \tau_k^{n-i}(s) \,\mathrm{d}s \bigg)
\end{equation}
and
\begin{equation}\label{eq: edo sigma^1 dual constant cn}
c_k^n = - \sum_{i=1}^{n-1} \int_0^1 \sigma_k^{n-i+1}(s) \tau_k^i(s) \,\mathrm{d}s - \int_0^1 \sigma_k^{1}(s) \tau_{k,0}^n(s) \,\mathrm{d}s . 
\end{equation}
\end{subequations}
Then we have, for all $k \in\mathbb{N}$,
\begin{subequations}\label{eq: edo sigma dual}
\begin{equation}\label{eq: edo sigma^1 dual}
(\tau_k^1)''+k^2\pi^2 \tau_k^1 = 0 , \qquad (\tau_k^1)'(0)=(\tau_k^1)'(1)=0 ,
\end{equation}
and
\begin{equation}\label{eq: edo sigma^k dual}
(\tau_k^n)''+k^2\pi^2 \tau_k^n = \sum_{i=1}^{n-1} \tilde{\nu}_k^i \tau_k^{n-i} , \qquad
(\tau_k^n)'(0)= 0 ,\quad (\tau_k^n)'(1) = - \tau_k^{n-1}(0) ,
\end{equation}
\end{subequations}
for all $n \geq 2$. Furthermore we have the biorthogonality condition:
\begin{equation}\label{eq: edo biortho cond}
\sum_{i=1}^{n} \int_0^1 \sigma_k^{n-i+1}(s) \tau_k^i(s) \,\mathrm{d}s = 0  \qquad \forall n \geq 2 .
\end{equation}
Finally, $\tau_k^n$ can be expressed as
\begin{equation}\label{eq: tauk polnomial * trig}
\tau_k^n(x) = \tilde{P}_k^n(x) \cos(k \pi x) + \tilde{Q}_k^n(x) \sin(k \pi x).
\end{equation}
for some polynomials $\tilde{P}_k^n,\tilde{Q}_k^n \in\mathbb{R}[X]$.
\end{lem}

\begin{proof}
First, \eqref{eq: edo sigma^1 dual} follows from the definitions \eqref{eq: def sigma_0^1 dual} and \eqref{eq: def sigma_k^1 dual} of $\tau_k^1$. The proof now consists, for $n \geq 2$, of applying iteratively Lemma~\ref{lemma: prel for identical heat equations - dual} by setting $g_0(x) = \sum_{i=1}^{n-2} \tilde{\nu}_i \tau_k^{n-i}$ and $g_1(x) = \tau_k^{1}$. It only has to be checked that $\int_0^1 \cos(k\pi s) g_1(s) \,\mathrm{d}s \neq 0$. For $k = 0$, due to \eqref{eq: def sigma_0^1 dual}, $\int_0^1 \sigma_0^1(s) \,\mathrm{d}s = 1$ and, for $k \geq 1$, due to \eqref{eq: def sigma_k^1 dual}, $\int_0^1 \cos(k \pi s) \sigma_k^1(s) \,\mathrm{d}s = \sqrt{2}/2$. The biorthogonality condition \eqref{eq: edo biortho cond} holds true by fixing the arbitrary constant $c$ appearing in Lemma~\ref{lemma: prel for identical heat equations - dual} as \eqref{eq: edo sigma^1 dual constant c0n} and \eqref{eq: edo sigma^1 dual constant cn}.
\end{proof}

\begin{rem}\label{rem: dual constants fixed for biorthogonality}
The construction of the Riesz basis $\Phi$ relied on the iterative application of Lemma~\ref{lemma: prel for identical heat equations}, giving Lemma~\ref{eq: integration in series trough heat equations}. In that construction, we fixed the arbitrary constant $c$ from Lemma~\ref{lemma: prel for identical heat equations} as $c = 0$. Note that for the dual construction of Lemma~\ref{eq: integration in series trough heat equations dual}, the arbitrary constant $c$ appearing in  Lemma~\ref{lemma: prel for identical heat equations - dual} still corresponds to the solutions to the homogeneous ODE, i.e., for $g_0=g_1=0$ and $\gamma = 0$. However, in the iterative application of Lemma~\ref{lemma: prel for identical heat equations - dual}, giving Lemma~\ref{eq: integration in series trough heat equations dual}, $c$  is not defined as $0$ but iteratively defined by \eqref{eq: edo sigma^1 dual constant c0n} and \eqref{eq: edo sigma^1 dual constant cn}. Such a definition of the $c$ constant is made to ensure that the biorthogonality condition \eqref{eq: edo biortho cond} holds. This latter condition will be key in the construction of the dual Riesz basis $\Psi$ of $\Phi$ to ensure that $\Phi$ and $\Psi$ are biorthogonal.
\end{rem}

\begin{rem}\label{rem: computation firsts tauk}
The functions $\tau_k^n$ can be iteratively computed using \eqref{eq: def sigma_0 dual} and \eqref{eq: def sigma_k dual}. For a given number $N$ of heat equations in the PDE cascade \eqref{eq: studied PDE cascade}, only the computation of $\tau_k^n$ for $1 \leq n \leq N$ is required. The first four computations give
\begin{align*}
\tau_0^{1}(x) & = 1 , \\
\tau_0^{2}(x) & = - \frac{x^2}{2} - \frac{1}{6} , \\
\tau_0^{3}(x) & = \frac{x^4}{24} - \frac{1}{15} , \\
\tau_0^{4}(x) & = - \frac{x^6}{720} + \frac{x^4}{144} + \frac{17x^2}{720} - \frac{457}{15120} 
\end{align*}
along with $\nu_0^n = \tilde{\nu}_0^n$ for $1 \leq n \leq 4$, and, for $k \geq 1$,
\begin{align*}
\tau_k^{1}(x) & = \sqrt{2} \cos(k\pi x) , \\
\tau_k^{2}(x) & = \frac{(-1)^{k+1} \sqrt{2} x}{k \pi} \sin(k\pi x) + \frac{4 - (-1)^k\sqrt{2}}{2k^2\pi^2} \cos(k\pi x) , \\
\tau_k^{3}(x) & = \frac{(3\sqrt{2}+4(-1)^{k+1})x}{2k^3\pi^3} \sin(k\pi x) 
+ \sqrt{2} \left( \frac{1-3x^2}{6k^2\pi^2} + 3 \frac{1+ (-1)^{k+1}\sqrt{2}}{k^4\pi^4} \right) \cos(k\pi x) , \\
\tau_k^{4}(x) & = \left( (-1)^{k} \sqrt{2} \frac{x(-2+x^2)}{6k^3\pi^3} + \frac{16+11\sqrt{2}(-1)^{k+1}}{2k^5\pi^5} x \right) \sin(k\pi x) \\
& \phantom{=}\; + \left( \frac{8 + 6\sqrt{2}(-1)^{k+1} + (-12+15\sqrt{2}(-1)^{k}) x^2}{12k^4\pi^4} + \sqrt{2} \frac{135\sqrt{2} + 162(-1)^{k+1}}{12k^6\pi^6} \right) \cos(k\pi x) , 
\end{align*}
with $\nu_k^n = \tilde{\nu}_k^n$ for $1 \leq n \leq 4$.
\end{rem}


\begin{lem}
Under Assumption~\ref{assump: identical heat equations}, the dual Riesz basis $\Psi = \{ \psi_{i,k} ,\mid\, 1 \leq i \leq N ,\; k \in\mathbb{N} \}$  of $\Phi$ is described by 
\begin{equation}\label{eq: mult eig - def psi}
\psi_{i,k} = ( \psi_{i,k}^j )_{1 \leq j \leq N}
= ( \underbrace{0 , \ldots , 0}_{i-1 \,\mathrm{times}} , \underbrace{\tau_k^1 , \tau_k^2 , \ldots , \tau_k^{N-i+1}}_{N-i+1\,\mathrm{elements}} ) \in\mathcal{D}(A^*) 
\end{equation}
for integers $1 \leq i \leq N$, $k \in\mathbb{N}$. Furthermore we have $\nu_k^i=\tilde{\nu}_k^i$ and 
\begin{equation}\label{eq: expression A* for gen eig vec}
\mathcal{A}^* \psi_{i,k}
= \lambda_{k} \psi_{i,k} + \sum_{j=1}^{N-i} \nu_k^{j} \psi_{i+j,k} .
\end{equation}
Hence, in $(\psi_{i,k})_{1 \leq i \leq N}$, the matrix of $\mathcal{A}$ is $M_k^\top$ as defined by \eqref{eq: matrix Mk}.
\end{lem}

\begin{proof}
By definition of $\mathcal{A}^*$ and the ODEs \eqref{eq: edo sigma dual} satisfied by $\tau_k^n$, we infer that 
\begin{equation}\label{eq: expression A* for gen eig vec - prel}
\mathcal{A}^* \psi_{i,k}
= \lambda_{k} \psi_{i,k} + \sum_{j=1}^{N-i} \tilde{\nu}_k^{j} \psi_{i+j,k} .
\end{equation}
We now show that $\Phi$ and $\Psi$ are biorthogonal, i.e., $\langle \phi_{i,k} , \psi_{i',k'} \rangle = \delta_{(i,k),(i',k')}$. Using the definitions \eqref{eq: mult eig - def phi} and \eqref{eq: mult eig - def psi} of $\phi_{i,k}$ and $\psi_{i,k}$, respectively, the fact that $\sigma_0^{1}(x) = \tau_0^1 (x) = 1$ while $\sigma_k^{1}(x) = \tau_{k}^1 (x) = \sqrt{2} \cos(k\pi x)$ for $k \in\mathbb{N}^*$, and the biorthogonality property \eqref{eq: edo biortho cond}, we infer that $\langle \phi_{i,k} , \psi_{i',k} \rangle = \delta_{i,i'}$. For $k \neq k'$, the computation of the two sides of the identity $\langle \mathcal{A} \phi_{i,k} , \psi_{i',k'} \rangle = \langle \phi_{i,k} , \mathcal{A}^* \psi_{i',k'} \rangle$, for $1 \leq i,i' \leq N$, by using \eqref{eq: expression A for gen eig vec} and \eqref{eq: expression A* for gen eig vec - prel} shows that $\langle \phi_{i,k} , \psi_{i',k'} \rangle = 0$. Hence $\Phi$ and $\Psi$ are biorthogonal. It remains to shows that $\nu_k^i=\tilde{\nu}_k^i$. This is inferred from the computation of the two sides of the identity $\langle \mathcal{A} \phi_{i,k} , \psi_{i',k} \rangle = \langle \phi_{i,k} , \mathcal{A}^* \psi_{i',k} \rangle$. 
\end{proof}

\subsection{Spectral reduction of the system}
Consider again the change of variable \eqref{eq: change of variable last component} that gives the homogenenous representation \eqref{eq: PDE homognenous coordinates}. Defining the states $\mathcal{X} = (y^j)_{1 \leq j \leq N}$ (original coordinates) and $\tilde{\mathcal{X}} = \left( (y^j)_{1 \leq j \leq N-1} , \tilde{y}^N \right)$ (homogeneous coordinates), we infer the change of variable formula \eqref{eq: change of variable abstract} along with the abstract system dynamics \eqref{eq: abstract PDE}. Define the coefficients of projection \eqref{eq: coeff of projection} and the vectors
\begin{equation*}
\bold{x}_k = \begin{bmatrix} x_{1,k} \\ \vdots \\ x_{N,k} \end{bmatrix} ,\quad
\tilde{\bold{x}}_k = \begin{bmatrix} \tilde{x}_{1,k} \\ \vdots \\ \tilde{x}_{N,k} \end{bmatrix} ,\quad
\alpha_k = \begin{bmatrix} \alpha_{1,k} \\ \vdots \\ \alpha_{N,k} \end{bmatrix} ,\quad
\beta_k = \begin{bmatrix} \beta_{1,k} \\ \vdots \\ \beta_{N,k} \end{bmatrix} .
\end{equation*}
We infer from \eqref{eq: change of variable abstract} that
$\tilde{\bold{x}}_k = \bold{x}_k + \beta_k u$
while \eqref{eq: abstract PDE} gives
\begin{equation}\label{eq: dynamics lambda_k homogeneous coordinates}
\dot{\tilde{\bold{x}}}_k = M_k \tilde{\bold{x}}_k + \alpha_k u + \beta_k v .
\end{equation}
Combining the two latter identities, we infer that 
\begin{equation}\label{eq: dynamics lambda_k}
\dot{\bold{x}}_k = M_k \bold{x}_k + N_k u
\end{equation}
where $N_k = \alpha_k + M_k \beta_k$. Hence, the dynamics of the mode $\lambda_k$ is described by \eqref{eq: dynamics lambda_k}. The study of its controllability properties relies on the following lemma. 

\begin{lem}\label{lem: control cond mult eig}
We have
$N_{k} = - \mu_k \begin{bmatrix} c_k^N & c_k^{N-1} & \ldots & c_k^1 & 1 \end{bmatrix}^\top$
for all $k \in\mathbb{N}$. Here $\mu_0 = 1$ and $\mu_k = \sqrt{2}$ for all $k \in\mathbb{N}^*$.
\end{lem}

\begin{proof}
By the definition \eqref{eq: coeff of projection} of $\alpha_{i,k}$ and using integrations by parts, we infer that
$$
\alpha_{i,k} = \int_0^1 (\varphi''+a \varphi) \psi_{i,k}^N \,\mathrm{d}x 
= \left[ \varphi' \psi_{i,k}^N - \varphi ( \psi_{i,k}^N )' \right]_0^1 + \int_0^1 \varphi ( (\psi_{i,k}^N)'' + a \psi_{i,k}^N ) \,\mathrm{d}x 
= - \psi_{i,k}^N(0) - \langle \beta , \mathcal{A}^* \psi_{i,k} \rangle
$$
where we have used that $\varphi(1)=\varphi'(1)=0$, $\varphi'(0)=1$, and $(\psi_{i,k}^N)'(0)=0$. In view of \eqref{eq: expression A* for gen eig vec} we have $\langle \beta , \mathcal{A}^* \psi_{i,k} \rangle = (M_k)_{L_i} \beta_k$ where $(M_k)_{L_i}$ denotes the $i$-th line of the matrix $M_k$, which gives $\alpha_{i,k}+(M_k)_{L_i} \beta_k=- \psi_{i,k}^N(0)$. Hence, for all $1 \leq i \leq N$, we have
$$
N_{k} = - \begin{bmatrix} \psi_{1,k}^N(0) & \psi_{2,k}^N(0) & \ldots \psi_{N,k}^N(0) \end{bmatrix}^\top 
= - \begin{bmatrix} \tau_{k}^N(0) & \tau_{k}^{N-1}(0) & \ldots \tau_{k}^{1}(0) \end{bmatrix}^\top
$$
where we have used \eqref{eq: mult eig - def psi}. The  conclusion follows from \eqref{eq: def sigma_0 dual}-\eqref{eq: def sigma_k dual}.
\end{proof}

Noting that $\nu_k^1 \neq 0$ and $\mu_k \neq 0$, the application of the Hautus test shows that \eqref{eq: dynamics lambda_k} is controllable for all $k \in\mathbb{N}$.

\subsection{Output feedback control strategy} 

Consider the output \eqref{eq: bounded measurement} measured on the last component $y^1$ of the PDE cascade \eqref{eq: studied PDE cascade}.

Let $\delta > 0$ be the targeted exponential decay rate for the closed-loop system. Using \eqref{eq: lambda_k}, we choose a large enough integer $n_{0} \geq 0$ such that $\lambda_{n_{0}} < - \delta$, implying that $\lambda_{k} \leq \lambda_{n_{0}} < - \delta$ for all integers $k \geq n_{0}$. We now introduce a finite-dimensional dynamics that captures the $n_{0}$ first modes of the plant. Defining the state vector
$X_1 = \mathrm{col}(\bold{x}_0,\bold{x}_1,\ldots,\bold{x}_{n_{0}-1})$,
we infer from \eqref{eq: dynamics lambda_k homogeneous coordinates} that
\begin{equation}\label{eq: finite dim dynamics first modes - mult eig}
\dot{X}_{1} = A_{1} X_{1} + B_{1,u} u + B_{1,v} v
\end{equation}
where
$$
A_{1} = \mathrm{diag}(M_{0},M_{1},\ldots,M_{n_{0}-1}) , \quad
B_{1,u} = \mathrm{col}( \alpha_{0} , \alpha_{1} , \ldots , \alpha_{n_{0}-1} ) , \quad
B_{1,v} = \mathrm{col}( \beta_{0} , \beta_{1} , \ldots , \beta_{n_{0}-1} ) .
$$
Recalling that $v = \dot{u}$, we define the augmented state vector
$X_{1,a} = \mathrm{col}(X_1,u)$,
obtaining the dynamics
\begin{equation}\label{eq: augmented finite dim dynamics first modes - mult eig}
\dot{X}_{1,a} = A_{1,a} X_{1,a} + B_{1,a} v
\end{equation}
where the matrices $A_{1,a}$  and $B_{1,a}$ are defined by \eqref{eq: augmented matrices}.

Under Assumption~\ref{assump: identical heat equations}, $k \neq l \Rightarrow \mathrm{sp}(M_k) \cap \mathrm{sp}(M_l) = \emptyset$. Therefore, the application of the Hautus test shows that the pair $(A_{1,a},B_{1,a})$ is controllable if and only if the last entry of $N_{k} = \alpha_{k} + M_{k} \beta_{k} \neq 0$ is not zero for all $0 \leq k \leq n_{0}-1$. The result of Lemma~\ref{lem: control cond mult eig} implies the following statement.

\begin{lem}
Under Assumption~\ref{assump: identical heat equations}, the pair $(A_{1,a},B_{1,a})$ is controllable.
\end{lem} 

\begin{rem}
At this step, it is possible to design a state-feedback control strategy as in Remark~\ref{rem: state-feedback}.
\end{rem}

Proceeding as in the previous section, the measurement \eqref{eq: bounded measurement} is expanded as
\begin{equation}\label{eq: output series expansion - mult eig}
z(t) = \sum_{i=1}^N \sum_{k \geq 0} c_{i,k} \tilde{x}_{i,k}(t) = C_1 X_1(t) + \sum_{k \geq n_{0}} c_{k} \tilde{\bold{x}}_{k}(t) 
\end{equation} 
where $c_{i,k}$ is defined by \eqref{eq: coeff projection measurement bounded operator} 
and
$$
c_{k} = \begin{bmatrix} c_{1,k} & c_{2,k} & \ldots & c_{N,k} \end{bmatrix} , \quad
C_1 = \begin{bmatrix} c_{0} & c_{1} & \ldots & c_{n_{0}-1} \end{bmatrix} .
$$
Again, under Assumption~\ref{assump: identical heat equations}, $k \neq l \Rightarrow \mathrm{sp}(M_k) \cap \mathrm{sp}(M_l) = \emptyset$. Hence, the application of the Hautus test shows that the pair $(A_1,C_1)$ is observable if and only if $c_{1,k} \neq 0$ for all $0 \leq k \leq n_{0}-1$. By the definition \eqref{eq: coeff projection measurement bounded operator} of $c_{i,k}$ and the expression of $\phi_{1,k}$ given by \eqref{eq: mult eig - def phi}, we infer the following result.

\begin{lem}\label{eq: mult eig obsv cond}
Under Assumption~\ref{asum: simple eigenvalues}, the pair $(A_1,C_1)$ is observable if and only if
\begin{equation}\label{eq: obsv cond mult eig}
\int_0^1 c(x) \cos(k \pi x) \,\mathrm{d}x \neq 0 , \quad 0 \leq k \leq n_{0}-1 .
\end{equation}
\end{lem}

We are now in position to define our output feedback control strategy. Let $n_{1} \geq n_{0}$ be an integer to be chosen large enough later on. We defined the following controller dynamics:
\begin{subequations}\label{eq: mult eig output feedback control strategy}
\begin{align}
\dot{u} & = v \\
\dot{\hat{X}}_1 & = A_1 \hat{X}_1 + B_{1,u} u + B_{1,v} v 
- L \left( C_1 \hat{X}_1 + \sum_{k = n_{0}}^{n_{1}} c_{k} \hat{\bold{x}}_{k} - z \right) \\
\dot{\hat{\bold{x}}}_{k} & = M_{k} \hat{\bold{x}}_{k} + \alpha_{k} u + \beta_{k} v  ,\; n_{0} \leq k \leq n_{1} \\
v & = -K_x \hat{X}_1 - k_u u 
\end{align}
\end{subequations}
where $k_u\in\mathbb{R}$ and $K_x \in\mathbb{R}^{1\times n_{0} N}$ are the feedback gains and $L \in \mathbb{R}^{n_{0} N}$ is the observer gain.


\begin{thm}\label{thm4}
Let $a_1=a_2=\ldots = a_N \in\mathbb{R}$ (Assumption~\ref{assump: identical heat equations}) and let $c \in L^2(0,1)$. Let $\delta > 0$ be arbitrary. Let $n_{0}\in\mathbb{N}$ be large enough so that $\lambda_{n_{0}} < - \delta$. Assuming that the observability condition \eqref{eq: obsv cond mult eig} holds, let $K = \begin{bmatrix} K_x & k_u \end{bmatrix}\in\mathbb{R}^{1\times(n_{0}N+1)}$ and $L \in \mathbb{R}^{n_{0}N}$ be matrices such that $A_{1,a}-B_{1,a}K$ and $A_1 - L C_1$ are Hurwitz with spectral abscissa less than $-\delta$. 

Then, for any $n_{1} \geq n_{0}$ large enough, there exists $C > 0$ such that, for all initial conditions $y_0^j \in L^2(0,1)$, $1 \leq j \leq N$, $u(0)=u_0\in\mathbb{R}$, $\hat{X}_1(0)\in\mathbb{R}^{n_{0}N}$, and $\hat{x}_{i,k}(0)\in\mathbb{R}$, the solutions of the closed-loop system consisting of the PDE cascade \eqref{eq: studied PDE cascade}, the output \eqref{eq: bounded measurement} and the controller \eqref{eq: mult eig output feedback control strategy}, satisfy \eqref{eq: stability estimate} with $Y = \mathcal{H}^0$ defined by \eqref{eq: Hilbert space H^0}, for all $t \geq 0$. 

If the initial condition is such that $y_0^j \in H^2(0,1)$ with $(y^j)'(0) = y^{j+1}(1)$ and $(y^j)'(1) = 0$  for all $1 \leq j \leq N-1$ while $(y^N)'(0) = u(0)$ and $(y^N)'(1) = 0$, the conclusion holds with $Y = \mathcal{H}^1$ defined by \eqref{eq: def space H^1}.
\end{thm}

\begin{proof}
We define $\hat{X}_1 = \mathrm{col}(\hat{\bold{x}}_{0},\hat{\bold{x}}_{1},\ldots,\hat{\bold{x}}_{n_{0}-1})$, $\bold{e}_{k} = \tilde{\bold{x}}_{k} - \hat{\bold{x}}_{k}$ and
\begin{align*}
E_1 &= \mathrm{col}( \bold{e}_{1} , \bold{e}_{2} , \ldots , \bold{e}_{n_{0}-1} ) , \quad
E_2 = \mathrm{col}( (n_{0}+1)^\kappa \bold{e}_{n_{0}} , \ldots , (n_{1}+1)^\kappa \bold{e}_{n_{1}} ) , \\
\hat{X}_{1,a} & = \mathrm{col}(\hat{X}^1,u) , \quad \hat{X}_2  = \mathrm{col}( \hat{\bold{x}}_{n_{0}} , \hat{\bold{x}}_{n_{0} + 1} , \ldots , \hat{\bold{x}}_{n_{1}} ) , \quad
X  = \mathrm{col}( \hat{X}_{1,a} , E_1 , \hat{X}_2 , E_2 ) .
\end{align*}
Here we set $\kappa = 0$. A different value will be set later on in the proof of Theorem~\ref{thm5}. Combining \eqref{eq: dynamics lambda_k homogeneous coordinates}, \eqref{eq: finite dim dynamics first modes - mult eig}, \eqref{eq: augmented finite dim dynamics first modes - mult eig}, \eqref{eq: output series expansion - mult eig}, and \eqref{eq: mult eig output feedback control strategy}, we obtain 
\begin{equation}\label{eq: dynamics truncated model mult eig}
\dot{X} = F X + \mathcal{L} \zeta
\end{equation}
where the measurement residue is defined by
\begin{equation}\label{eq: residues of measurement mult eig}
\zeta = \sum_{k \geq n_{1}+1} c_{k} \tilde{\bold{x}}_{k} 
\end{equation}
and 
$$
F = \begin{bmatrix}
A_{1,a} - B_{1,a} K & L_a C_1 & 0 & L_a C_2 \\
0 & A_1-LC_1 & 0 & -L C_2 \\
\left[ 0 \; B_{2,u} \right] - B_{2,v} K & 0 & A_2 & 0 \\
0 & 0 & 0 & A_2
\end{bmatrix} , \quad
\mathcal{L} = \mathrm{col}( L_a , - L , 0 , 0 ) ,
$$
where
\begin{align*}
K & = \begin{bmatrix} K_x & k_u \end{bmatrix} , \quad
L_a = \mathrm{col}(L,0) , \quad
A_2 = \mathrm{diag}(M_{n_{0}},M_{n_{0} + 1},\ldots,M_{n_{1}}) , \\
B_{2,u} & = \mathrm{col}(\alpha_{n_{0}},\alpha_{n_{0} + 1},\ldots,\alpha_{n_{1}}) , \quad
B_{2,v} = \mathrm{col}(\beta_{n_{0}},\beta_{n_{0} + 1},\ldots,\beta_{n_{1}}) , \\
C_2 &= \begin{bmatrix}  \frac{c_{n_{0}}}{(n_{0}+1)^\kappa } & \frac{c_{n_{0} + 2}}{(n_{0}+1)^\kappa } & \ldots & \frac{c_{n_{1}}}{(n_{1}+1)^\kappa } \end{bmatrix} .
\end{align*}
The input $u$ and its time derivative $v$ are expressed as \eqref{eq: express u and v}.

Since, for $k \geq n_{0}$, the matrix $M_k$ defined by \eqref{eq: matrix Mk} has a spectral abscissa less than $-\delta$, it follows that $\tilde{M_k} = M_k + \delta I$ is Hurwitz. We can thus define $P_k \succ 0$ as the solution of the Lyapunov equation  
\begin{equation}\label{eq: Lyapunov equation}
M_k^\top P_k + P_k M_k + 2 \delta P_k = \tilde{M}_k^\top P_k + P_k \tilde{M}_k = - \vert \lambda_k \vert I
\end{equation}
for all $k \geq n_{0}$. Defining $\Delta_k = P_k - \frac{1}{2} I$, let us prove that $\Vert \Delta_k \Vert = \mathrm{O}(1/k^2)$ as $k \rightarrow + \infty$. 
Using \eqref{eq: asymptotic estimate sigma - 12} and \eqref{eq: matrix Mk}, the matrix $R_k = M_k - \lambda_k I$ satisfies $\Vert R_k \Vert = \mathrm{O}(1)$. Substituting $\tilde{M}_k = M_k + \delta I = (- \vert \lambda_k \vert + \delta) I + R_k$ and $P_k = \frac{1}{2} I + \Delta_k$ into \eqref{eq: Lyapunov equation} while using \eqref{eq: lambda_k}, we infer that $\Vert \Delta_k \Vert = \mathrm{O}(1/k^2)$. Hence there exist $\gamma_m,\gamma_M > 0$ such that
\begin{equation}\label{eq: bounds on solution to Lyap equation}
\Vert P_k \Vert \leq \gamma_M 
\quad\textrm{and}\quad
\gamma_m I \preceq P_k \preceq \gamma_M I  \quad \forall k \geq n_{0} +1 . 
\end{equation}
Noting that $R_k \in\mathbb{R}^{N \times N}$ is nilpotent (because upper triangular with zeros on the diagonal), we obtain $e^{(M_k + \delta I) t} = e^{(\lambda_k+\delta) t} e^{R_k t} = e^{(\lambda_k+\delta) t} \sum_{i=0}^{N-1} \frac{R_k^i t^i}{i!}$. Since $\Vert R_k \Vert = \mathrm{O}(1)$, we infer the existence of a polynomial function $P_\infty \in \mathbb{R}_{N-1}[X]$, not depending on $k$, such that $\Vert e^{(M_k + \delta I) t} \Vert \leq P_\infty(t) e^{(\lambda_k+\delta) t} \leq P_\infty(t) e^{(\lambda_{n_0}+\delta) t}$ for all $t \geq 0$ and $k \geq n_0$. Since $\lambda_{n_0} < -\delta$, there exist $c,\epsilon > 0$, not depending on $n_1$, such that 
\begin{equation}\label{eq: estimate exp A_2 t}
\Vert e^{(A_2+\delta I) t} \Vert \leq c e^{-\epsilon t}  \quad \forall t \geq 0  \quad \forall n_1 \geq n_0 .
\end{equation}
We now define the Lyapunov functional candidates to establish the stability estimate \eqref{eq: stability estimate}. In $\mathcal{H}^0$ norm, we define
\begin{equation}\label{eq: output-feedback Lyap H0 mult eig}
V = X^\top P X + \sum_{k \geq n_{1}+1} \tilde{x}_{k}^\top P_k \tilde{x}_{k}
\end{equation}  
for some matrix $P \succ 0$. The equivalence between \eqref{eq: output-feedback Lyap H0 mult eig} and the (squared) $\mathcal{H}^0$ norm follows from \eqref{eq: bounds on solution to Lyap equation} and Lemma~\ref{eq: mult eig - Riesz basis}. In $\mathcal{H}^1$ norm, we define for classical solutions
\begin{equation}\label{eq: output-feedback Lyap H1 mult eig}
V = X^\top P X + \sum_{k \geq n_{1}+1} (1+k^2) \tilde{x}_{k}^\top P_k \tilde{x}_{k}
\end{equation}
for some matrix $P \succ 0$.  The equivalence between \eqref{eq: output-feedback Lyap H1} and the (squared) $\mathcal{H}^1$ norm follows from \eqref{eq: bounds on solution to Lyap equation} and Lemma~\ref{lem: mult eig Riesz basis H1} since the series expansion \eqref{eq: series expansion H1 norm} holds. 

In what follows we focus on the $\mathcal{H}^1$ norm. We compute the derivative of $V$ defined by \eqref{eq: output-feedback Lyap H1 mult eig} along the solutions of \eqref{eq: dynamics lambda_k homogeneous coordinates} and \eqref{eq: dynamics truncated model mult eig}. Using the Lyapunov equations \eqref{eq: Lyapunov equation}, the estimates \eqref{eq: bounds on solution to Lyap equation}, Young's inequality and \eqref{eq: express u and v}, we obtain
	\begin{align*}
	\dot{V}+2\delta V 
	&= \tilde{X}^\top 
	\begin{bmatrix} F^\top P + P F + 2 \delta P & P\mathcal{L} \\  
	\mathcal{L}^\top P & 0 \end{bmatrix}	 
	\tilde{X} 
	+ 2 \sum_{k \geq n_{1} +1} (1+k^2) \tilde{x}_k^\top (M_k^\top P_k + P_k M_k + 2 \delta P_k ) \tilde{x}_k \\
	& \phantom{=}\; + 2 \sum_{k \geq n_{1} +1} (1+k^2) \tilde{x}_{k}^\top P_k ( \alpha_{k} u + \beta_{k} v ) \\
	& \leq \tilde{X}^\top 
	\begin{bmatrix} \Theta_{1,1} & P\mathcal{L} \\  
	\mathcal{L}^\top P & 0 \end{bmatrix}	
	\tilde{X} 
	+ 2 \sum_{k \geq n_{1} +1} (1+k^2) \left( \lambda_{k} + \frac{\gamma_M^2(1+k^2)}{\epsilon} \right) \Vert \tilde{x}_{k} \Vert^2 
	\end{align*}
	for all $\epsilon > 0$, where $\tilde{X} = \mathrm{col}(X,\zeta)$ and
	\begin{align*}
	\Theta_{1,1} & = F^\top P + P F + 2 \delta P + \epsilon \left( S_{\alpha,n_{1}} \tilde{E}^\top \tilde{E} + S_{\beta,n_{1}} \tilde{K}^\top \tilde{K} \right)
	\end{align*} 
	where $S_{\alpha,n_{1}} = \sum_{k \geq n_{1} +1} \Vert \alpha_{k} \Vert^2 < \infty$ and $S_{\beta,n_{1}} = \sum_{k \geq n_{1} +1} \Vert\beta_{k}\Vert^2 < \infty$. Using \eqref{eq: residues of measurement mult eig} and the Cauchy-Schwarz inequality, we infer that 
	\begin{equation*}
		\zeta^2 \leq \underbrace{\sum_{k \geq n_{1}+1} \Vert c_{k} \Vert^2}_{= S_{\zeta,n_{1}} < \infty} \sum_{k \geq n_{1}+1} \Vert \tilde{x}_{k} \Vert^2 .
	\end{equation*}	
	 Combining the two latter estimates, we obtain that
	\begin{align}
	\dot{V}+2\delta V 
	& \leq \tilde{X}^\top \Theta_1 \tilde{X} +\sum_{k \geq n_{1} +1 } (1+k^2) \Gamma_{k} \Vert \tilde{x}_{k} \Vert^2 \label{eq: dotV mult eig}
	\end{align}
	where
	\begin{subequations}
	\begin{align}
	\Theta_1 & = \begin{bmatrix} \Theta_{1,1} & P\mathcal{L} \\  
	\mathcal{L}^\top P & -\eta \end{bmatrix} , \label{eq: Theta_1 mult eig} \\
	\Gamma_{k} & = 2 \left( \lambda_{k} + \frac{\gamma_M^2 (1+k^2)}{\epsilon} \right) + \frac{\eta S_{\zeta,n_{1}}}{1+k^2} 
	= 2 \left( - k^2 \left( \pi^2 - \frac{\gamma_M^2}{\epsilon} \right) + a + \frac{\gamma_M^2}{\epsilon} \right) + \frac{\eta S_{\zeta,n_{1}}}{1+k^2} \label{eq: Gamma_k mult eig}
	\end{align}
	\end{subequations}
	for all $\epsilon,\eta > 0$. In particular, for all $\epsilon > \gamma_M^2/\pi^2$, we have $\Gamma_{k} \leq \Gamma_{n_{1}+1}$ for all $k \geq n_{1}+1$. We thus infer that
	\begin{align*}
	\dot{V}+2\delta V 
	& \leq \tilde{X}^\top \Theta_1 \tilde{X} 
	+ \Gamma_{n_{1}+1} \sum_{k \geq n_{1}+1} (1+k^2) \Vert \tilde{x}_{k} \Vert^2
	\end{align*}
	for all $\epsilon > \gamma_M^2/\pi^2$ and all $\eta > 0$. Hence $\dot{V}+2\delta V \leq 0$, giving the claimed stability estimate \eqref{eq: stability estimate}, provided that there exists an integer $n_{1} \geq n_{0} + 1$, real numbers $\epsilon > \gamma_M^2/\pi^2$ and $\eta > 0$ and a matrix $P \succ 0$ such that
	$\Theta_1 \preceq 0$ and $\Gamma_{n_{1}+1} \leq 0$.
	
	Let us prove the feasibility of these constraints. Recall that $F$ is Hurwitz with spectral abscissa less than $-\delta$. Note also that $\Vert C_2 \Vert = \mathrm{O}(1)$, $\Vert B_{2,u} \Vert = \mathrm{O}(1)$ and $\Vert B_{2,v} \Vert = \mathrm{O}(1)$ as $n_{1} \rightarrow + \infty$. Finally, recall that \eqref{eq: estimate exp A_2 t} holds for some constants $c,\epsilon > 0$ not depending on $n_1$. Hence, the application of the lemma in Appendix of \cite{lhachemi2020finite} shows that the solution $P \succ 0$ to $F^\top P + P F + 2 \delta P = -I$ satisfies $\Vert P \Vert = \mathrm{O}(1)$ as $n_{1} \rightarrow + \infty$. We now arbitrarily fix $\epsilon > \gamma_M^2/\pi^2$ and define
	\begin{equation*}
\eta = \left\{\begin{array}{cl}
\frac{1}{\sqrt{S_{\zeta,n_{1}}}} & \mathrm{if}\; S_{\zeta,n_{1}} \neq 0 \\
n_{1} & \mathrm{otherwise}
\end{array}\right. .
\end{equation*}
	We infer that $\Gamma_{n_{1} + 1} \rightarrow - \infty$ as $n_{1} \rightarrow +\infty$. Moreover, the Schur complement shows that  $\Theta_1 \preceq 0$ if $n_{1}$ is chosen sufficiently large. This completes the proof.
\end{proof}

\subsection{Extension to pointwise Dirichlet measurement}

Consider the case of the pointwise measurement of $y^1$ given by \eqref{eq: Dirichlet measurement}. Then the series expansion \eqref{eq: output series expansion} still holds true but with $c_{i,k} = \phi_{i,k}^1(\xi_p)$. We thus infer, similarly to Lemma~\ref{eq: mult eig obsv cond} and in view of the expression of $\phi_{i,k}^1$ given by \eqref{eq: mult eig - def phi}, that the pair $(A_1,C_1)$ is observable if and only if
\begin{equation}\label{eq: obsv cond for Dirichlet mult eig}
\cos(k\pi\xi_p) \neq 0 , \quad 0 \leq k \leq N_{0,1}-1 .
\end{equation}

\begin{thm}\label{thm5}
Considering the pointwise measurement \eqref{eq: Dirichlet measurement} for some $\xi_p \in [0,1]$, replacing the observability condition \eqref{eq: obsv cond mult eig} by \eqref{eq: obsv cond for Dirichlet mult eig}, the  statement of Theorem~\ref{thm4} holds true with $Y = \mathcal{H}^0$ defined by \eqref{eq: Hilbert space H^0}.
\end{thm}

The proof follows the arguments developed in the proof of Theorems~\ref{thm2} and~\ref{thm4} by setting $\kappa = 1$. 

\begin{rem}
In the case of $N$ identical heat equations (Assumption~\ref{assump: identical heat equations}), the pointwise measurement of $y^1_x$ given by \eqref{eq: Neumann measurement} is never an appropriate output for the feedback stabilization of the PDE cascade \eqref{eq: studied PDE cascade}. Indeed, in view of \eqref{eq: lambda_k}, the open-loop system is unstable if and only if $\lambda_0 = a > 0$. But this mode is never observable for the pointwise measurement \eqref{eq: Neumann measurement} because $c_{1,0} = (\phi_{1,0}^1)'(\xi_p) = 0$. Hence, an output feedback strategy based on the pointwise measurement of $y^1_x$ given by \eqref{eq: Neumann measurement} is not sufficient to stabilize the PDE cascade \eqref{eq: studied PDE cascade}. 
\end{rem}

\section{Conclusion}\label{sec: conclusion}
We have studied the boundary output feedback stabilization problem for a cascade on $N$ heat equations, with a spectral approach, using the key property that the (generalized) eigenelements of the underlying operator form a Riesz basis. Our presentation focused on the two extreme case of either two-by-two distinct eigenvalues or eigenvalues of multiplicity $N$, but the intermediate cases can be handled similarly. Future work may address other PDE cascades.

\paragraph{Acknowlegment.}
The third author acknowledges the support of ANR-20-CE40-0009 (TRECOS).

\bibliographystyle{abbrv}        
\bibliography{elsart-bib}

\begin{thebibliography}{10}

\bibitem{auriol2020output}
J.~Auriol.
\newblock Output feedback stabilization of an underactuated cascade network of
  interconnected linear pde systems using a backstepping approach.
\newblock {\em Automatica}, 117:108964, 2020.

\bibitem{auriol2024output}
J.~Auriol.
\newblock Output-feedback stabilization of an underactuated network of $ n $
  interconnected $ n+ m $ hyperbolic pde systems.
\newblock {\em IEEE Transactions on Automatic Control}, 2024.

\bibitem{bhandari2021boundary}
K.~Bhandari and F.~Boyer.
\newblock Boundary null-controllability of coupled parabolic systems with
  {R}obin conditions.
\newblock {\em Evolution Equations and Control Theory}, 10(1):61--102, 2021.

\bibitem{camacho2020boundary}
L.~Camacho-Solorio, R.~Vazquez, and M.~Krstic.
\newblock Boundary observers for coupled diffusion--reaction systems with
  prescribed convergence rate.
\newblock {\em Systems \& Control Letters}, 135:104586, 2020.

\bibitem{chen2017backstepping}
S.~Chen, R.~Vazquez, and M.~Krstic.
\newblock Backstepping control design for a coupled hyperbolic-parabolic mixed
  class {PDE} system.
\newblock In {\em IEEE 56th Annual Conference on Decision and Control}, pages
  664--669, 2017.

\bibitem{chowdhury2023boundary}
S.~Chowdhury, R.~Dutta, and S.~Majumdar.
\newblock Boundary controllability and stabilizability of a coupled first-order
  hyperbolic-elliptic system.
\newblock {\em Evolution Equations \& Control Theory}, 12(3), 2023.

\bibitem{coron2004global}
J.-M. Coron and E.~Tr{\'e}lat.
\newblock Global steady-state controllability of one-dimensional semilinear
  heat equations.
\newblock {\em SIAM Journal on Control and Optimization}, 43(2):549--569, 2004.

\bibitem{curtain2012introduction}
R.~F. Curtain and H.~Zwart.
\newblock {\em An introduction to infinite-dimensional linear systems theory},
  volume~21.
\newblock Springer Science \& Business Media, 2012.

\bibitem{ghousein2020backstepping}
M.~Ghousein and E.~Witrant.
\newblock Backstepping control for a class of coupled hyperbolic-parabolic
  {PDE} systems.
\newblock In {\em 2020 American Control Conference (ACC)}, pages 1600--1605.
  IEEE, 2020.

\bibitem{gohberg1978introduction}
I.~Gohberg and M.~G. Kreuin.
\newblock {\em Introduction to the theory of linear nonselfadjoint operators},
  volume~18.
\newblock American Mathematical Soc., 1978.

\bibitem{kang2016stabilisation}
W.~Kang and B.-Z. Guo.
\newblock Stabilisation of unstable cascaded heat partial differential equation
  system subject to boundary disturbance.
\newblock {\em IET Control Theory \& Applications}, 10(9):1027--1039, 2016.

\bibitem{leiva2003lemma}
H.~Leiva.
\newblock A lemma on {$C_0$}-semigroups and applications.
\newblock {\em Quaestiones Mathematicae}, 26(3):247--265, 2003.

\bibitem{lhachemi2020finite}
H.~Lhachemi and C.~Prieur.
\newblock Finite-dimensional observer-based boundary stabilization of
  reaction-diffusion equations with either a {D}irichlet or {N}eumann boundary
  measurement.
\newblock {\em Automatica}, 135:109955, 2022.

\bibitem{lhachemi2020pi}
H.~Lhachemi, C.~Prieur, and E.~Tr{\'e}lat.
\newblock {PI} regulation of a reaction--diffusion equation with delayed
  boundary control.
\newblock {\em IEEE Transactions on Automatic Control}, 66(4):1573--1587, 2020.

\bibitem{lhachemi2025boundary}
H.~Lhachemi, C.~Prieur, and E.~Tr{\'e}lat.
\newblock Boundary control of heat-heat cascades.
\newblock {\em Under review}, 2025.

\bibitem{lhachemi2025controllability}
H.~Lhachemi, C.~Prieur, and E.~Tr{\'e}lat.
\newblock Controllability and stabilization of a wave-heat cascade system.
\newblock {\em Under review}, 2025.

\bibitem{rosier2013unique}
L.~Rosier and B.-Y. Zhang.
\newblock Unique continuation property and control for the
  {B}enjamin--{B}ona--{M}ahony equation on a periodic domain.
\newblock {\em Journal of Differential Equations}, 254(1):141--178, 2013.

\bibitem{russell1978controllability}
D.~L. Russell.
\newblock Controllability and stabilizability theory for linear partial
  differential equations: recent progress and open questions.
\newblock {\em {SIAM} Review}, 20(4):639--739, 1978.

\bibitem{sakawa1983feedback}
Y.~Sakawa.
\newblock Feedback stabilization of linear diffusion systems.
\newblock {\em SIAM journal on control and optimization}, 21(5):667--676, 1983.

\bibitem{tang2024boundary}
J.-Q. Tang, J.-M. Wang, and W.~Kang.
\newblock Boundary feedback stabilization of an unstable cascaded heat--heat
  system with different reaction coefficients.
\newblock {\em Systems \& Control Letters}, 183:105684, 2024.

\bibitem{tang2025sampled}
J.-Q. Tang, J.-M. Wang, and W.~Kang.
\newblock Sampled-data control of an unstable cascaded heat--heat system with
  different reaction coefficients.
\newblock {\em Automatica}, 171:111904, 2025.

\bibitem{trelat2024control}
E.~Tr{\'e}lat.
\newblock {\em Control in finite and infinite dimension}.
\newblock Springer, 2024.

\bibitem{vazquez2016boundary}
R.~Vazquez and M.~Krstic.
\newblock Boundary control of coupled reaction-advection-diffusion systems with
  spatially-varying coefficients.
\newblock {\em IEEE Transactions on Automatic Control}, 62(4):2026--2033, 2016.

\bibitem{wang2015stabilization}
J.-M. Wang, L.-L. Su, and H.-X. Li.
\newblock Stabilization of an unstable reaction--diffusion {PDE} cascaded with
  a heat equation.
\newblock {\em Systems \& Control Letters}, 76:8--18, 2015.

\bibitem{xu2023stabilization}
X.~Xu, L.~Liu, M.~Krstic, and G.~Feng.
\newblock Stabilization of chains of linear parabolic {PDE--ODE} cascades.
\newblock {\em Automatica}, 148:110763, 2023.

\bibitem{young2001introduction}
R.~M. Young.
\newblock {\em An Introduction to Non-Harmonic Fourier Series, Revised Edition,
  93}.
\newblock Elsevier, 2001.

\bibitem{zhang2004polynomial}
X.~Zhang and E.~Zuazua.
\newblock Polynomial decay and control of a {1-D} hyperbolic--parabolic coupled
  system.
\newblock {\em Journal of Differential Equations}, 204(2):380--438, 2004.

\end{thebibliography}

\end{document}